\documentclass[sn-mathphys-num]{sn-jnl}

\usepackage{graphicx}%
\usepackage{multirow}%
\usepackage{amsmath,amssymb,amsfonts}%
\usepackage{amsthm}%
\usepackage{mathrsfs}%
\usepackage[title]{appendix}%
\usepackage{xcolor}%
\usepackage{textcomp}%
\usepackage{manyfoot}%
\usepackage{booktabs}%
\usepackage{algorithm}%
\usepackage{algorithmicx}%
\usepackage{algpseudocode}%
\usepackage{listings}%

\theoremstyle{thmstyleone}%
\newtheorem{theorem}{Theorem}
\theoremstyle{thmstyletwo}%
\newtheorem{example}{Example}%

\theoremstyle{thmstylethree}%
\newtheorem{definition}{Definition}%

\newtheorem{corollary}{Corollary}
\usepackage{ytableau}
\usepackage{array}
\usepackage{caption}
\usepackage{latexsym}
\usepackage{placeins}
\newcommand{\seqnum}[1]{\href{http://oeis.org/#1}{\underline{#1}}}
\DeclareMathOperator {\mypod}{pod}
\DeclareMathOperator {\myped}{ped}

\begin{document}

\title[$s$-Modular, $s$-congruent and $s$-duplicate partitions]{$s$-Modular, $s$-congruent and $s$-duplicate partitions}

\author[1,4,2]{\fnm{Mohammed L.} \sur{Nadji}}\email{m.nadji@usthb.dz}

\author*[3,4]{\fnm{Moussa} \sur{Ahmia}}\email{moussa.ahmia@univ-jijel.dz}
\equalcont{These authors contributed equally to this work.}

\affil[1]{\orgdiv{Faculty of Mathematics}, \orgname{University of Science and Technology Houari Boumediene}, \state{Algiers}, \country{Algeria}}

\affil[2]{\orgdiv{RECITS laboratory}, \orgaddress{\street{BP 32}, \city{El Alia, Bab Ezzouar}, \postcode{16111}, \state{Algiers}, \country{Algeria}}}

\affil*[3]{\orgdiv{Department of Mathematics}, \orgname{University of Mohamed Seddik Ben Yahia}, \state{Jijel}, \country{Algeria}}

\affil*[4]{\orgdiv{LMAM laboratory}, \orgaddress{\street{BP 98}, \city{Ouled Aissa}, \postcode{18000}, \state{Jijel}, \country{Algeria}}}

\abstract{In this paper, we investigate the combinatorial properties of three classes of integer partitions: (1) $s$-modular partitions, a class consisting of partitions into parts with a number of occurrences (i.e., multiplicity) congruent to $0$ or $1$ modulo $s$, (2) $s$-congruent partitions, which generalize Sellers' partitions into parts not congruent to $2$ modulo $4$, and (3) $s$-duplicate partitions, of which the partitions having distinct odd parts and enumerated by the function $\mypod(n)$ are a special case. In this vein, we generalize Alladi's series expansion for the product generating function of $\mypod(n)$ and show that Andrews' generalization of G\"ollnitz-Gordon identities coincides with the number of partitions into parts simultaneously $s$-congruent and $t$-distinct (parts appearing fewer than $t$ times).}



\maketitle

\section{Introduction} 
\label{Intro}
A \textit{partition} $\lambda$ of a positive integer $n$ is a finite sequence of positive integers that sum to $n$. We write the partition as $\lambda_{1}^{u_{1}}+ \lambda_{2}^{u_{2}}+\dots+ \lambda_{k}^{u_{k}}=n$, where each part $\lambda_{i}$ occurs $u_{i}$ times (referred to as the multiplicity of $\lambda_{i}$) with $\lambda_{i}, u_{i} \geq 1$ for $1 \leq i \leq k$.  Let $\mid \lambda \mid$ denote the sum of the parts and $\ell(\lambda)$ denote the total number of parts in the partition $\lambda$. For instance, $2^{2}+1^{2}$ is a partition of $6$ with $\ell( \lambda)=4$.  Here and throughout we will use the standard $q$-series notations
\begin{displaymath}
(a;q)_{n}:=\begin{cases} 
      \prod\limits_{i=0}^{n-1}(1-aq^{i}) &  \text{if $n>0$,} \\
      1 & \text{if $n=0$.}
    \end{cases}
\end{displaymath}
Moreover,
\begin{align*}
(a;q)_{\infty}&=\lim_{n\rightarrow \infty}(a;q)_{n} \ \text{for} \  |q|<1,\\
(a_{1}, a_{2},\dots, a_{j};q)_{\infty}&=(a_{1};q)_{\infty}(a_{2};q)_{\infty}\cdots (a_{j};q)_{\infty}.
\end{align*}

Throughout the history of integer partitions, most studies have focused on partitions with restrictions on the parts. Examples include partitions into odd (or even) parts \cite{r14, r16}, $\ell$-regular partitions \cite{r6, r7}, partitions into parts that belong to certain residue classes modulo some positive integer \cite{r4, r17}, and more. Nevertheless, the multiplicities (i.e., number of occurrences) of parts have not received significant attention in the literature on integer partitions, as reflected in the limited number of related papers. Many of these few papers are dedicated to the partitions where parts appear fewer than $m$ times \cite{r11, r12} (the number of such partitions is equal to the number of $m$-regular partitions).

In this paper, we consider congruences of the multiplicities of parts and introduce the $s$-modular partitions. Additionally, considerations of congruences of the multiplicities lead to the definition of $s$-congruent and $ s$-duplicate partitions. Throughout the remainder of the paper, $s$ is always considered to be an even positive integer greater than or equal to $4$.
\begin{definition} A partition $\lambda$ is $s$-modular if every part occurs with multiplicity congruent to  $0$ or $1$ modulo $s$. We denote by $M_{s}(n)$ the number of $s$-modular partitions of $n$.
\end{definition}
\begin{definition} A partition $\lambda$ is said to be $s$-congruent if the parts are not congruent to $2,4,6,\dots,(s-2)$ modulo $s$. We denote by $C_{s}(n)$ the number of $s$-congruent partitions of $n$.
\end{definition}
\begin{definition}  A partition $\lambda$ is $s$-duplicate if every part in the partition with multiplicity greater than one is congruent to $0$ modulo $s/2$. We denote by $D_{s}(n)$ the number of $s$-duplicate partitions of $n$.
\end{definition}
Note that the $s$-congruent partitions were introduced by Ballantine and Welch \cite{BallantineWelch2023} as partitions in which all even parts are congruent to $0$ modulo $s$. 

Through this study, we present several combinatorial properties for the three classes of partitions. We establish connections between them and show that they satisfy a number of recurrence relations. Additionally, we highlight the links between some of these classes of partitions and several celebrated theorems in the theory of partitions. 

\emph{Ramanujan's general theta-function} $f(a,b)$ \cite[p. 34, 18.1]{Berndt1991} is defined as
\[ 
f(a,b):=\sum_{n=-\infty}^{\infty}a^{n(n+1)/2}b^{n(n-1)/2}, \quad \mid ab \mid <1.
\]
From $f(a,b)$ we have
\[ 
\psi(q):=f(q,q^{3})
=\sum_{n=0}^{\infty}q^{n(n+1)/2}
=\dfrac{(q^{2},q^{2})_{\infty}}{(q,q^{2})_{\infty}},
\]
and
\[
\psi(-q)=f(-q,-q^{3})
=\sum_{n=0}^{\infty}(-q)^{n(n+1)/2}
=\dfrac{(q,q)_{\infty}(q^{4},q^{4})_{\infty}}{(q^{2},q^{2})_{\infty}},
\]
where the product representations arise from famous Jacobi’s triple product identity \cite[p. 35, Entry 19]{Berndt1991}
\[
f(a,b)=(-a,ab)_{\infty}(-b,ab)_{\infty}(ab,ab)_{\infty}.
\]

In the special case when $s=4$, the $s$-duplicate become what is known in the literature as POD partitions, partitions wherein odd parts are distinct and even parts are
unrestricted. The number of POD partitions of $n$ is denoted by $\mypod(n)$. They appear frequently in the literature, for example in works of Andrews \cite{r2, r3}, and Berkovich and Garvan \cite{r5}, where the number $\mypod(n)$ satisfies the generating function
\[
\sum\limits_{n\geq 0}\mypod(n)q^{n}=\dfrac{1}{\psi(-q)}=\dfrac{(q^{2},q^{2})_{\infty}}{(q,q)_{\infty}(q^{4},q^{4})_{\infty}}.
\]
The aforementioned function has various combinatorial interpretations \cite[\seqnum{A006950}]{r21}, such as in Schr\"oder partitions and lattice paths. Also, it is interesting to note that this enumeration function is relevant from an algebraic point of view: $\mypod(n)$ equals the number of nilpotent conjugacy classes in the Lie algebras of skew-symmetric $n\times n$ matrices. This intriguing connection suggests that this function and its combinatorial interpretations play a substantial role in representation theory and raise bigger questions for the general case of $s$-sduplicate partitions and their impact on these topics.

We can say very little about $D_{s}(n)$ from an arithmetic point of view, except in the most restricted cases when $s=4$. It was first studied arithmetically by Hirschhorn and Sellers \cite{r13}. In their paper, they derived an infinite family of Ramanujan-type congruences for any integers $n,\alpha\geq 0$, such that
\[
D_{4}(3^{2\alpha+3}n+(23\times 2^{2\alpha +2}+1)/8)\equiv 0\ (\mathrm{mod}\ 3).
\]
Later on, Radu and Sellers \cite{r19} established several new congruences, such as
\begin{align*}
&D_{4}(135n+8)\equiv 0\ (\mathrm{mod}\ 5),\\
&D_{4}(135n+107)\equiv 0\ (\mathrm{mod}\ 5),\\
&D_{4}(567n+260)\equiv 0\ (\mathrm{mod}\ 7),\\
&D_{4}(675n+647)\equiv 0\ (\mathrm{mod}\ 25),\\
&D_{4}(3375n+1997)\equiv 0\ (\mathrm{mod}\ 125).
\end{align*}
In 2015, Wang \cite{r23}, relying heavily on results from Lovejoy and Osburn \cite{r15}, presented some new additional congruences for $D_{4}(n)$ modulo $3,5$, and $9$. If $p$ is a prime with $p\equiv 1$ (mod $3$), and $n,\alpha \geq 0$ such that $pn\equiv 5$ (mod $8$), then
\begin{align*}
&D_{4}((3p^{6\alpha+5}n+1)/8)\equiv 0\ (\mathrm{mod}\ 3),\\
&D_{4}(5^{2\alpha+2}n+(11\cdot 5^{2\alpha+1}+1)/8)\equiv 0\ (\mathrm{mod}\ 5),\\
&D_{4}((3p^{18\alpha+17}n+1)/8)\equiv 0\ (\mathrm{mod}\ 9).
\end{align*}
Moreover, $D_{4}(n)$ finds mention in numerous works across different fields, including those of  Sills \cite{r20}, Ferrari \cite{r10}, Drake \cite{r8}, and Zaletel and Mong \cite{r24}. This extensive coverage looks very promising for the general case or at least some other restricted cases, as it may lead to significant contributions and find diverse applications in various topics. It is worth noting that the OEIS \cite{r21} sequence \seqnum{A096981} corresponds to the function $C_{6}(n)$, and unlike $D_{4}(n)$, it still lacks arithmetic studies, although it makes appearances in works such as Spector \cite{r22}.

The \textit{Ferrers diagram} is a method of geometric representation of partitions using rows of boxes. Each row represents the value of the part, where the rows are ordered in a non-increasing order. In the literature, this is also referred to as \textit{Young diagram}. The $k$-modular Ferrers diagram, which is a modification of the Ferrers diagram, represents a partition in such a way that each part is depicted by a left-justified row  of $k$'s with an $r$ at the right end, where $1\leq r \leq k$. For example, the following figure illustrates the $3$-modular Ferrers diagram of the partition $(14^{2}, 12^{2}, 7, 6, 5, 4, 3)$ with the diagram of $(5^{2}, 4^{2}, 3, 2^{3}, 1)$.
\begin{figure} [h!]
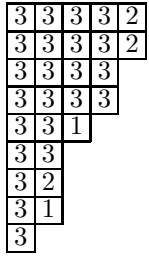

\begin{center}
\ytableausetup
{mathmode, boxframe=normal, boxsize=1em}
\begin{ytableau}
 3 & 3 & 3 & 3 & 2\\
 3 & 3 & 3 & 3 & 2\\
 3 & 3 & 3 & 3\\
 3 & 3 & 3 & 3\\
 3 & 3 & 1 \\
 3 & 3 \\
 3 & 2 \\
 3 & 1 \\
 3 \\
\end{ytableau}
\caption{The $3$-modular Ferrers diagram of $(14^{2}, 12^{2}, 7, 6, 5, 4, 3)$.}
\end{center}
\end{figure}

One of the most fundamental identities in the theory of partitions and $q$-series is
\begin{equation}
\sum_{n \geq 0}\dfrac{q^{n(n+1)/2}(1+bq)(1+bq^{2})\cdots(1+bq^{n})}{(1-q)(1-q^{2})\cdots(1-q^{n})} = \prod_{m\geq 1} \dfrac{(1+bq^{2m})}{(1-q^{2m-1})} \label{eqLebesgue}
\end{equation}
due to Lebesgue \cite[Chap. 2]{ANDB}. The right hand-side generating function of $\eqref{eqLebesgue}$ enumerates the number of partitions into distinct even parts and unrestricted odd parts (the function $\myped(n)$), where the exponent on $b$ keeps track the number of even parts. By setting $b=-1$ we obtain the famous Gauss identity
\begin{equation}
\sum_{n \geq 0} q^{n(n+1)/2} = \prod_{m\geq 1} \dfrac{(1-q^{2m})}{(1-q^{2m-1})}, \label{eqGauss}
\end{equation}
which is as significant as  Euler's celebrated Pentagonal Numbers Theorem
\begin{equation}
(q;q)_{\infty} = 1+\sum_{k\geq 1} (-1)^{k} ( q^{(3k^{2}-k)/2} + q^{(3k^{2}+k)/2}).
 \label{eqPNT}
\end{equation}
Sylvester \cite{Syl} also presented an important identity by analyzing partitions into distinct parts via Durfee squares (the largest square that is contained within a partition's Ferrers diagram). The identity is given by
\begin{equation}
\prod_{n \geq 1} (1+bq^{n}) = 1+\sum_{k\geq 1} \dfrac{b^{k}q^{(3k^{2}-k)/2}(-bq;q)_{k-1}(1+bq^{2k})}{(q;q)_{k}}.
 \label{eqSylvester}
\end{equation}
The case $b=-1$ in $\eqref{eqSylvester}$ is Euler's Pentagonal Numbers Theorem. Fine \cite{Fin} defined the function
\begin{equation}
F(\alpha,\beta,\tau;q)= \sum_{n\geq 0} \dfrac{(\alpha q;q)_{n}}{(\beta q;q)_{n}} \tau^{n} \label{eqfine}
\end{equation}
and studied carious transformations and iterations for $F(\alpha,\beta,\tau;q)$. One of the obtained results is the Rogers-Fine identity
\begin{equation}
F(\alpha,\beta,\tau;q)= \sum_{n\geq 0} \dfrac{(\alpha q;q)_{n} (\alpha\tau q/\beta;q)_{n}(1-\alpha \tau q^{2n+1})}{(\beta q;q)_{n} (\tau;q)_{n+1}} \beta^{n}\tau^{n}q^{n^{2}}. \label{eqRogers-Fine}
\end{equation}
Andrews \cite{r3} provided a combinatorial proof of $\eqref{eqRogers-Fine}$ by studying partitions in terms of $2n \times n$ Durfee rectangles. Another proof of $\eqref{eqRogers-Fine}$ has been presented by Zeng \cite{Zen} by utilizing Sylvester's bijection connecting partitions into odd parts with partitions into distinct parts.

Alladi \cite{r1} presented a series expansion with different parameters for the product generating function of  $\mypod(n)$ using $2$-modular Ferrers diagrams and their Durfee squares. Several fundamental identities in the theory of partitions and $q$-series, including those of Sylvester, Lebesgue, Gauss, and Rogers-Fine, emerge as special cases of this series expansion.  The expansion is given by
\begin{equation}
\dfrac{(-bq;q^{2})_{\infty}}{(cq^{2};q^{2})_{\infty}}=1+\sum_{k\geq 1} \dfrac{c^{k} q^{2k^{2}-1} (-bq;q^{2})_{k-1} (-bc^{-1}q;q^{2})_{k-1} (bc^{-1}+q) (1+bq^{4k-1})}{(cq^{2};q^{2})_{k} (q^{2};q^{2})_{k}}, \label{eqAlladi}
\end{equation}
where the powers of $b$ and $c$ keep track of the number of odd and even parts respectively. For example, we demonstrate the relation between the series expansion for the product generating function of $\mypod(n)$ and Sylvester's identity, as discussed by Alladi \cite{r1}. In $\eqref{eqAlladi}$, note that
\begin{equation}
c^{k}(-bc^{-1}q;q^{2})_{k-1}(bc^{-1}+q)=(c+bq)(c+bq^{3})\cdots (c+bq^{2k-3})(b+cq). \label{AlladiSylvester}
\end{equation}
By setting $c=0$ in \eqref{AlladiSylvester}, the righthand side of \eqref{AlladiSylvester} equals
\[
b^{k}q^{(k-1)^{2}}.
\]
Thus, if $c=0$ in \eqref{eqAlladi}, we obtain
\begin{equation}
(-bq;q^{2})_{\infty}= 1+\sum_{k\geq 1} \dfrac{b^{k}q^{3k^{2}-2k} (-bq;q^{2})_{k-1}(1+bq^{4k-1})}{(q^{2};q^{2})_{k}}. \label{AlladiSylvester1}
\end{equation}
Sylvester's identity $\eqref{eqSylvester}$ follows from $\eqref{AlladiSylvester1}$ by the substitutions $b\rightarrow bq$ and then $q^{2} \rightarrow q$. Alladi \cite{All1} also provided a combinatorial interpretation in terms of partitions for Sylvester's identity $\eqref{eqSylvester}$, demonstrating its equivalence to \textbf{Theorem $\ref{thm1}$}.
\begin{theorem} \label{thm1}  (Alladi).  Let $D$ denote the set of partitions into distinct parts, and $D_{3}$ the set of partitions into parts differing by at least $3$. Let $\nu_{d}(\lambda)$ be the number of different parts of $\lambda$. Also, for $ \lambda^{*}\in D_{3}$, $\mid \lambda^{*}\mid=\lambda^{*}_{1}+\lambda^{*}_{2}+\dots+\lambda^{*}_{k}$, let $\nu_{3}(\lambda^{*})$ denote the number of strict inequalities $\lambda^{*}_{i}-\lambda^{*}_{i+1}>3$, for $i=1,2,\dots,k$, where $\lambda^{*}_{k+1}=-1$. Then we have
\[
\sum_{\lambda\in D, \mid\lambda\mid=n} c^{\ell(\lambda)} = \sum_{\lambda^{*}\in D_{3}, \mid\lambda^{*}\mid=n} c^{\ell(\lambda^{*})}(1+c)^{\nu_{3}(\lambda^{*})}.
\] 
\end{theorem}
The series expansion $\eqref{eqAlladi}$ also establishes links with G\"ollnitz's deep theorem. Alladi \cite{All2} presented the equivalence between the three parameter refinement of \textbf{Theorem $\ref{thm1}$} and the three parameter generalization and refinement of G\"ollnitz's (Big) partition theorem \cite{r4}. Notably, the identity $\eqref{eqAlladi}$ holds an advantage in providing more flexibility in choosing specializations, which are used to derive several fundamental identities. To extend its applicability, we provide a generalization by employing the concept of $k$-modular Ferrers diagrams.

We adopt the convention of identifying partitions with their multisets of parts, arranged in decreasing order. To enhance clarity and readability, we employ informal language. For example, when we state that "the function $F(q)$ generates partitions with property $X$", this means that $F(q)$ is the generating function for the number of partitions of $n$ possessing property $X$. 

Note that if $a(n)$ denotes the number of partitions with certain properties, then we define $a(n) := 0$ for any $n$ that is a negative integer. 

We shall use $\mathbb{M}_{s}(n)$ (resp. $\mathbb{C}_{s}(n)$ and $\mathbb{D}_{s}(n)$) to denote the set of $s$-modular (resp. $s$-congruent and $s$-duplicate) partitions of $n$, where $\mid \mathbb{M}_{s}(n) \mid =M_{s}(n)$ (resp. $\mid \mathbb{C}_{s}(n) \mid  = C_{s}(n)$ and $\mid  \mathbb{D}_{s}(n) \mid  = D_{s}(n)$).

In combinatorial proofs, we establish a bijective function $f: \mathbb{A} \rightarrow \mathbb{B}$ between two sets of partitions $\mathbb{A}$ and $\mathbb{B}$, creating a one-to-one correspondence where $\mathbb{A}$ and $\mathbb{B}$ are finite sets of partitions of fixed size. Each part $\lambda_{i}$ of a partition $\lambda\in \mathbb{A}$ is mapped to either a single part $\beta_{i}$ or a sequence of parts $\beta_{i_{1}}, \beta_{i_{2}}, \dots, \beta_{i_{m}}$ in a partition $\beta \in \mathbb{B}$. Thus, the bijective function $f$ and its inverse $f^{-1}$ operate on the individual parts of $\lambda$ and $\beta$, respectively. In other words, we map a multiset of parts to another multiset of parts, all summing to $n$, such that for $\lambda=(\lambda_{1}^{u_{1}}, \lambda_{2}^{u_{2}}, \dots, \lambda_{k}^{u_{k}})$ and $\beta=(\beta_{1}^{w_{1}}, \beta_{2}^{w_{2}}, \dots, \beta_{k}^{w_{k}})$ , we have
\begin{align*}
\beta &=f(\lambda)=f\bigr((\lambda_{1}^{u_{1}}, \dots, \lambda_{k}^{u_{k}})\bigr):=\bigr( f(\lambda_{1}^{u_{1}}), \dots, f(\lambda_{k}^{u_{k}}) \bigr) = \bigr(\cup_{i=1}^{k} f(\lambda_{i}^{u_{i}}) \bigr),\\
\lambda &=f^{-1}(\beta)=f^{-1}\bigr( (\beta_{1}^{w_{1}}, \dots, \beta_{k}^{w_{k}}) \bigr):=\bigr( f^{-1}(\beta_{1}^{w_{1}}), \dots, f^{-1}(\beta_{k}^{w_{k}}) \bigr)= \bigr( \cup_{i=1}^{k} f^{-1}(\beta_{i}^{w_{i}}) \bigr).
\end{align*}
Here and throughout, we omit the parentheses around the parts forming the image of $f(\lambda_{i}^{u_{i}})$  (resp. $f^{-1}(\beta_{i}^{w_{i}})$). The image of $f(\lambda_{i}^{u_{i}})$ (resp. $f^{-1}(\beta_{i}^{w_{i}})$) could result a various types of partitions. For instance, $f(\lambda_{i}^{u_{i}})=u_{i}^{\lambda_{i}}$, meaning the value of the part becomes the multiplicity, and the value of the multiplicity becomes the part. Similarly, $f(\lambda_{i}^{u_{i}})=u_{i}\lambda_{i}$ denotes that the resulting image is the product of the part by its multiplicity. Other examples include $f(\lambda_{i}^{u_{i}})=(u_{i}-1)^{\lambda_{i}}, \lambda_{i}$ or $f(\lambda_{i}^{u_{i}})=(s\lambda_{i}/2)^{2(u_{i}-1)/s}, \lambda_{i}$, where the obtained image is a sequence of parts. In the case when $(u_{i}-1)=0$ (i.e., $u_{i}=1$), we exclude that part from the sequence, such that $f(\lambda_{i}^{u_{i}})=\lambda_{i}$.

\section{Combinatorial properties}
\label{combprop}
In the following results, we introduce the generating functions and several Combinatorial properties for the three classes of partitions.

\begin{theorem} \label{thmmodular} The generating function for $M_{s}(n)$, the number of $s$-modular partitions of $n$, is given by
\[
\sum_{n\geq 0} M_{s}(n) q^{n} = \dfrac{(-q;q)_{\infty}}{(q^{s};q^{s})_{\infty}}.
\]
\end{theorem}

\begin{proof} Since
\[
\prod_{n\geq 1}(1+q^{sn}+q^{2sn}+q^{3sn}+\cdots)=\dfrac{1}{(q^{s};q^{s})_{\infty}}
\]
is the generating function for the number of partitions where the multiplicity of every part is congruent to $0$ modulo $s$, then
\[
\sum_{n\geq 0} M_{s}(n) q^{n} = \prod_{n\geq 1}(1+q^{n})(1+q^{sn}+q^{2sn}+q^{3sn}+\cdots)=\dfrac{(-q;q)_{\infty}}{(q^{s};q^{s})_{\infty}},
\]
from which our result follows.
\end{proof}

For example, $M_{4}(8)=10$ and the corresponding set of partitions is
\[
\mathbb{M}_{4}(8)=\lbrace (8), (7, 1), (6, 2), (5, 3), (5, 2, 1), (4, 3, 1), (3, 1^{5}), (4, 1^{4}), (2^{4}), (1^{8}) \rbrace.
\]
\begin{theorem} \label{thmcongruent} The generating function for $C_{s}(n)$, the number of $s$-congruent partitions of $n$, is given by
\[
\sum_{n\geq 0} C_{s}(n) q^{n} = \dfrac{1}{(q;q^{2})_{\infty}(q^{s};q^{s})_{\infty}}.
\]
\end{theorem}

\begin{proof} The generating function $1/(q^{s};q^{s})_{\infty}$ generates partitions into parts congruent to $0$ modulo $s$. Similarly, the known generating function $1/(q;q^{2})_{\infty}$ generates partitions into odd parts. By multiplying both generating functions together, we obtain the desired result.
\end{proof}
For example, $C_{6}(8)=7$ such that
\[
\mathbb{C}_{6}(8)=\lbrace (7, 1), (6, 1^{2}), (5, 3), (5, 1^{3}), (3^{2}, 1^{2}), (3, 1^{5}), (1^{8}) \rbrace.
\]

\begin{theorem} \label{thmduplicate} The generating function for $D_{s}(n)$, the number of $s$-duplicate partitions of $n$, satisfies the identity
\[
\sum_{n\geq 0} D_{s}(n) q^{n} = \dfrac{(-q;q)_{\infty}}{(-q^{s/2};q^{s/2})_{\infty} (q^{s/2};q^{s/2})_{\infty}}.
\]
\end{theorem}

\begin{proof} The generating function $(-q;q)_{\infty}/(-q^{s/2};q^{s/2})_{\infty}$ generates partitions into distinct parts indivisible by $s/2$, and on the other hand, $1/(q^{s/2};q^{s/2})_{\infty}$ generates unrestricted partitions into parts multiples of $s/2$.
\end{proof}
For example, 
\[
\mathbb{D}_{6}(6)=\lbrace (6), (5,1), (4,2), (3,2,1), (3^{2})  \rbrace,
\]
from which we conclude that $D_{6}(6)=5$.

In the following theorem, we lay the foundations of the connection between $\mathbb{M}_{s}(n)$, $\mathbb{C}_{s}(n)$, and $\mathbb{D}_{s}(n)$ for $n\geq 0$.
\begin{theorem} \label{thmequivalence} For every positive integer $n\geq 0$,
\[
M_{s}(n)=C_{s}(n)=D_{s}(n).
\]
\end{theorem}

\begin{proof} We give both a generating function proof and a combinatorial proof of the result. By employing the facts $(-q;q)_{\infty}=(q;q^{2})^{-1}_{\infty}$ and $(1-q^{sn})=(1+q^{sn/2})(1-q^{sn/2})$ (since $s$ is even and $\geq 4$) for every $n \geq 1$ in \textbf{Theorem $\ref{thmmodular}$}, we obtain
\begin{eqnarray*}
\sum\limits_{n\geq 0}M_{s}(n)q^{n} &=& \prod_{n\geq 1} \dfrac{(1+q^{n})}{(1-q^{sn})}\\
&=& \prod_{n\geq 1} \dfrac{1}{(1-q^{2n-1})(1-q^{sn})}= \sum\limits_{n\geq 0}C_{s}(n)q^{n}\\
&=& \prod_{n\geq 1} \dfrac{(1+q^{n})}{(1+q^{sn/2})(1-q^{sn/2})}= \sum\limits_{n\geq 0}D_{s}(n)q^{n}.
\end{eqnarray*}

Now, we provide combinatorial proofs of the three parts of the theorem.

\textbf{The bijection $\mathbb{M}_{s}(n)\Leftrightarrow \mathbb{C}_{s}(n)$ when $s=2^{p}$ for $p\geq 2$:}

Let $\lambda=(\lambda_{1}^{u_{1}}, \lambda_{2}^{u_{2}}, \dots, \lambda_{k}^{u_{k}})\in \mathbb{M}_{s}(n)$ and define the map $f:\mathbb{M}_{s}(n)\rightarrow \mathbb{C}_{s}(n)$ by $\lambda \rightarrow f(\lambda)=\bigcup_{i=1}^{k} f(\lambda_{i}^{u_{i}})$  as follows. If $\lambda_{i} \equiv 2,4,6, \dots, (s-2)$ (mod $s$), then for $1\leq i \leq k$, we have
\begin{displaymath}
f(\lambda_{i}^{u_{i}})=\begin{cases} 
     u_{i}^{\lambda_{i}} & \text{if  $u_{i}\equiv\ 0$ (mod $s$),} \\
      (u_{i}-1)^{\lambda_{i}}, \ell_{i}^{2^{ r_{i}}} & \text{if  $u_{i}\equiv\ 1$ (mod $s$),}
    \end{cases}
\end{displaymath}
where each part $\lambda_{i}$ can be expressed uniquely as $\lambda_{i}=2^{r_{i}} \ell_{i}$ with $\ell_{i}$ odd. Note that if $\lambda \in \mathbb{M}_{s}(n)$, then $\lambda_{i }\equiv \pm 2^{r_{i}}$ (mod $2^{r_{i}+2}$). Else, if $\lambda_{i} \not\equiv 2,4,6, \dots, (s-2)$ (mod $s$), then 
\[
f(\lambda_{i}^{u_{i}})=\lambda_{i}^{u_{i}}.
\]
As an illustration, for $(n,s)=(18,8)$ and the partitions $\lambda=(4, 3 , 2, 1^{9})$ and $\beta=(5, 4, 1^{9})$, we have
\[
f(\lambda) \rightarrow \Biggl\{ \begin{array}{l}
f(4):  4=2^{2} \cdot 1 \rightarrow f(4)=(1-1)^{4}, 1^{4} \rightarrow f(4)=1^{4}\\
f(3)=3 \ \text{since} \ 3 \not\equiv 0 \ (\text{mod} \ 8) \ \text{and} \ f(1^{9})=1^{9} \ \text{since} \ 1 \not\equiv 0 \ (\text{mod} \ 8)\\
f(2):  2=2\cdot 1 \rightarrow f(2)=(1-1)^{2}, 1^{2} \rightarrow f(2)=1^{2}
\end{array}
\]
\[
f(\beta) \rightarrow \Biggl\{ \begin{array}{l}
f(5)=5 \ \text{since} \ 5 \not\equiv 0 \ (\text{mod} \ 8)\\
f(4):  4=2^{2} \cdot 1 \rightarrow f(4)=(1-1)^{4}, 1^{4} \rightarrow f(4)=1^{4}\\
f(1^{9})=1^{9} \ \text{since} \ 1 \not\equiv 0 \ (\text{mod} \ 8).
\end{array}
\]

Let $\lambda=(\lambda_{1}^{u_{1}}, \lambda_{2}^{u_{2}}, \dots, \lambda_{k}^{u_{k}})\in \mathbb{C}_{s}(n)$ and define the inverse map $f^{-1}:\mathbb{C}_{s}(n)\rightarrow \mathbb{M}_{s}(n)$ by $\lambda \rightarrow f^{-1}(\lambda)=\bigcup_{i=1}^{k} f^{-1}(\lambda_{i}^{u_{i}})$. If $\lambda_{i}\equiv 0$ (mod $s$), then for $1\leq i \leq k$, we have
\begin{displaymath}
f^{-1}(\lambda_{i}^{u_{i}})=\begin{cases} 
     \lambda_{i}^{u_{i}} & \text{if  $u_{i}\equiv\ 0,1$ (mod $s$),} \\
      u_{i}^{\lambda_{i}} & \text{if  $u_{i}\equiv\ 2, 4, 6, \dots, (s-2)$ (mod $s$),} \\
      (u_{i}-1)^{\lambda_{i}}, \lambda_{i} & \text{if  $u_{i}\equiv\ 3, 5,7, \dots, (s-1)$ (mod $s$).}
    \end{cases}
\end{displaymath}
Else, if $\lambda_{i} \not\equiv 0 \pmod{s}$, define the vector $W_{p-1}$ as
\[
W_{p-1} = (2a_{1}, 2^{2}a_{2}, \dots, 2^{p-1}a_{p-1}),
\]
where $a_{i}\in \lbrace 0,1 \rbrace$ and $i\in \lbrace 1,2,\dots (p-1) \rbrace$. The values of $a_{i}$ are chosen in such a way that $W_{p-1}$ attains a unique magnitude $\mid W_{p-1} \mid$ for which 
\[
m=u_{i}-\mid W_{p-1} \mid \equiv 0,1 \ (\mathrm{mod} \ s),
\]
where $m$ is the maximum possible value less than $u_{i}$. Consequently, we obtain the expansion 
\[
\lambda_{i}^{u_{i}}= \lambda_{i}^{m} + \lambda_{i}^{2a_{1}}+ \dots + \lambda_{i}^{2^{p-1}a_{p-1}},
\]
where if $a_{i}=0$ or $m=0$, we exclude the corresponding part $\lambda_{i}$ from the expansion. Then, the inverse map $f^{-1}(\lambda_{i}^{u_{i}})$ is given by 
\begin{displaymath}
f^{-1}(\lambda_{i}^{u_{i}})=\begin{cases} 
     \lambda_{i}^{u_{i}} & \text{if  $u_{i}\equiv\ 0,1$ (mod $s$),} \\
     \lambda_{i}^{m}, 2a_{1}\lambda_{i}, 2^{2}a_{2}\lambda_{i}, \dots ,2^{p-1}a_{p-1}\lambda_{i} & \text{if  $u_{i}\not\equiv\ 0,1$ (mod $s$).}
    \end{cases}
\end{displaymath}
As an illustration, for $(n,s)=(18,8)$ and the partitions $\mu=(3,1^{15})$ and $\gamma=(5,1^{13})$, we have
\[
f^{-1}(\mu) \rightarrow \Biggl\{ \begin{array}{l}
f^{-1}(3)=3 \ \text{since} \ 3 \not\equiv 0 \ (\text{mod} \ 8) \ \text{and} \ u_{1}=1\\
f^{-1}(1^{15}) : m=15 - \mid(2\cdot 1,2^{2}\cdot 1)\mid\equiv 1 \ (\text{mod} \ 8) \rightarrow f^{-1}(1^{15})=1^{9},2,4
\end{array}
\]
\[
f^{-1}(\gamma) \rightarrow \Biggl\{ \begin{array}{l}
f^{-1}(5)=5 \ \text{since} \ 5 \not\equiv 0 \ (\text{mod} \ 8) \ \text{and} \ u_{1}=1\\
f^{-1}(1^{13}): m=13 - \mid(2 \cdot 0,2^{2} \cdot 1)\mid=9\equiv 1 \ (\text{mod} \ 8) \rightarrow f^{-1}(1^{13})=1^{9},4.
\end{array}
\]

\textbf{The bijection $\mathbb{M}_{s}(n)\Leftrightarrow \mathbb{D}_{s}(n)$ when $s=2^{p}$ for $p\geq 2$:}

Let $\lambda=(\lambda_{1}^{u_{1}}, \lambda_{2}^{u_{2}}, \dots, \lambda_{k}^{u_{k}})\in \mathbb{M}_{s}(n)$ and define the map $g:\mathbb{M}_{s}(n)\rightarrow \mathbb{D}_{s}(n)$ by $\lambda \rightarrow g(\lambda)=\bigcup_{i=1}^{k} g(\lambda_{i}^{u_{i}})$. If $\lambda_{i}$ is either a multiple of $s/2$, or a non-multiple of $s/2$ with $u_{i}=1$, then 
\[
g(\lambda_{i}^{u_{i}})=\lambda_{i}^{u_{i}}.
\]
Else, if $\lambda_{i}$ is a non-multiple of $s/2$ with $u_{i}>1$, then let the unique expansion of the multiplicity $u_{i}= b_{0}+2^{p}a_{0}+2^{p+1}a_{1}+\dots+2^{p+r}a_{r}$ with $a_{i}\in \lbrace 0,1 \rbrace$ and $r\geq 0$. Therefore, we obtain the transformation 
\[
\lambda_{i}^{u_{i}}= (\lambda_{i}^{b_{0}}, \lambda_{i}^{2^{p}a_{0}}, \lambda_{i}^{2^{p+1} a_{1}}, \dots , \lambda_{i}^{2^{p+r}a_{r}}).
\]
Then, for $0 \leq j \leq r$, let $u'_{j}:= 2^{p+j} a_{j}$ and define $g(\lambda_{i}^{u'_{i}})$ by
\begin{displaymath}
g(\lambda_{i}^{u'_{i}})=\begin{cases} 
     (u'_{i}\lambda_{i}/2)^{2} & \text{if  $\lambda_{i}$ is odd,} \\
     (u'_{i}\lambda_{i}/4)^{4} & \text{if   $\lambda_{i}$ is even,}
    \end{cases}
\end{displaymath}
and if $b_{0}=1$, then 
\[
g(\lambda_{i}^{b_{0}})=\lambda_{i}.
\]
As an illustration, for $(n,s)=(10,4)$ and the partitions $\lambda=(3,2,1^{5})$ and $\beta=(2,1^{8})$, we have
\[
g(\lambda) \rightarrow \Biggl\{ \begin{array}{l}
g(3)=3 \ \text{since} \ u_{1}=1 \\
g(2)=2 \ \text{since} \ 2 \equiv 0 \ (\mathrm{mod} \ 2)\\
g(1^{5}) : 5=1+2^{2} \cdot 1\rightarrow 1^{5}=1,1^{4}\rightarrow g(1^{5})=1,2^{2}
\end{array}
\]
\[
g(\beta) \rightarrow \Biggl\{ \begin{array}{l}
g(2)=2 \ \text{since} \ 2 \equiv 0 \ (\mathrm{mod} \ 2)\\
g(1^{8}) : 8=0+2^{2}\cdot 0+2^{3} \cdot 1\rightarrow 1^{8}=1^{8}\rightarrow g(1^{8})=4^{2}.
\end{array}
\]

Let $\lambda=(\lambda_{1}^{u_{1}}, \lambda_{2}^{u_{2}}, \dots, \lambda_{k}^{u_{k}})\in \mathbb{D}_{s}(n)$ and define the map $g^{-1}:\mathbb{D}_{s}(n)\rightarrow \mathbb{M}_{s}(n)$ by $\lambda \rightarrow g^{-1}(\lambda)=\bigcup_{i=1}^{k} g^{-1}(\lambda_{i}^{u_{i}})$. If $\lambda_{i}$ is either a multiple of $s/2$ with $u_{i}\equiv 0,1$ (mod $s$), or a non-multiple of $s/2$ with $u_{i}=1$, then 
\[
g^{-1}(\lambda_{i}^{u_{i}})=\lambda_{i}^{u_{i}}.
\]
Now, if $\lambda_{i}$ is a multiple of $s/2$ with $u_{i}\not\equiv 0,1$ (mod $s$), in this case every part $\lambda_{i}$ can be expressed uniquely as $\lambda_{i}=2^{r_{i}}\ell_{i}$, where $\ell_{i}\geq 1$ is odd and $r_{i}\geq 1$. Then, for $n\in\lbrace 2,4 \rbrace$, we obtain the expansion $\lambda_{i}^{u_{i}}=\lambda_{i}^{u_{i}-n}+\lambda_{i}^{n}$ such that $u_{i}-n\equiv 0,1$ (mod $s$). Therefore, we have
\[
g^{-1}(\lambda_{i}^{u_{i}})=\lambda_{i}^{u_{i}-n}, (n\ell_{i}/2)^{2^{r_{i}+1}}.
\]
Notice that if $u_{i}-n=0$, then the part $\lambda_{i}^{u_{i}-n}$ is excluded from the resulting image. As an illustration, for $(n,s)=(10,4)$ and the partitions $\mu=(3,2^{3},1)$ and $\gamma=(4^{2},2)$, we have
\[
g^{-1}(\mu) \rightarrow \Biggl\{ \begin{array}{l}
g^{-1}(3)=3\\
g^{-1}(2^{3}) : 2=2^{1}\cdot 1 \rightarrow 1=3-2 \rightarrow 2^{3}=2+2^{2}\rightarrow g^{-1}(2^{3})=2,1^{4}\\
g^{-1}(2)=2
\end{array}
\]
\[
g^{-1}(\gamma) \rightarrow \Biggl\{ \begin{array}{l}
g^{-1}(4^{2}) : 4=2^{2}\cdot  1 \rightarrow 0=2-2 \rightarrow g^{-1}(4^{2})=1^{8}\\
g^{-1}(2)=2.
\end{array}
\]

\textbf{The bijection $\mathbb{M}_{s}(n)\Leftrightarrow \mathbb{C}_{s}(n)$ when $s\neq2^{p}$ for $p\geq 2$:}

Let $\lambda=(\lambda_{1}^{u_{1}}, \lambda_{2}^{u_{2}}, \dots, \lambda_{k}^{u_{k}})\in \mathbb{M}_{s}(n)$ and define the map $h:\mathbb{M}_{s}(n)\rightarrow \mathbb{C}_{s}(n)$ by $\lambda \rightarrow h(\lambda)=\bigcup_{i=1}^{k} h(\lambda_{i}^{u_{i}})$. If $\lambda_{i} \equiv 2,4,\dots,(s-2)$ (mod $s$), then 
\begin{displaymath}
h(\lambda_{i}^{u_{i}})=\begin{cases} 
     u_{i}^{\lambda_{i}} &\text{if  $u_{i}\equiv 0$ (mod $s$),} \\
     (u_{i}-1)^{\lambda_{i}}, \ell_{i}^{2^{r_{i}}} &\text{if  $u_{i}\equiv 1$ (mod $s$),}
    \end{cases}
\end{displaymath}
where $\lambda_{i}=2^{r_{i}} \ell_{i}$ for odd $\ell_{i}\geq 1$ and $r_{i}\geq 1$. If $\lambda_{i} \not\equiv 2,4,\dots,(s-2)$ (mod $s$), then
\begin{displaymath}
h(\lambda_{i}^{u_{i}})=\begin{cases} 
     (s\lambda_{i}/2)^{2u_{i}/s} & \text{if  $u_{i}\equiv 0$ (mod $s$),} \\
     (s\lambda_{i}/2)^{2(u_{i}-1)/s}, \lambda_{i}  & \text{if $u_{i}\equiv 1$ (mod $s$).}
    \end{cases}
\end{displaymath}
As an illustration, for $(n,s)=(18,6)$ and the partition $\lambda=(10, 2, 1^{6})$, we have
\[
h(\lambda) \rightarrow \Biggl\{ \begin{array}{l}
h(10): 10 \equiv 4 \ ( \mathrm{mod} \ 6), u_{1}=1, \ \text{and} \  10=2 \cdot 5 \rightarrow h(10)=5^{2} \\
h(2) : 2 \equiv 2 \ ( \mathrm{mod} \ 6), u_{2}=1, \ \text{and} \  2=2 \cdot 1 \rightarrow h(2)=1^{2}\\
h(1^{6}): 1 \not\equiv 2,4 \ ( \mathrm{mod} \ 6) \ \text{and} \ u_{3}=6 \rightarrow h(1^{6})=3^{2}.
\end{array}
\]

Let $\lambda=(\lambda_{1}^{u_{1}}, \lambda_{2}^{u_{2}}, \dots, \lambda_{k}^{u_{k}})\in \mathbb{C}_{s}(n)$ and define the inverse map $h^{-1}:\mathbb{C}_{s}(n)\rightarrow \mathbb{M}_{s}(n)$ by $\lambda \rightarrow h^{-1}(\lambda)=\bigcup_{i=1}^{k} h^{-1}(\lambda_{i}^{u_{i}})$. If $\lambda_{i} \equiv 0$ (mod $s$), then
\begin{displaymath}
h^{-1}(\lambda_{i}^{u_{i}})=\begin{cases} 
     u_{i}^{\lambda_{i}} & \text{if  $u_{i}$ is even,} \\
     (u_{i}-1)^{\lambda_{i}},\lambda_{i} & \text{if  $u_{i}$ is odd.}
    \end{cases}
\end{displaymath}
Else, if $\lambda_{i}\not\equiv 0$ (mod $s$), then when $\lambda_{i}=1$, let the unique expansion of the multiplicity $u_{i}=a_{0}+2a_{1}+2^{2}a_{2}+\dots+2^{r}a_{r}$  where $r\geq 0$ and $a_{i}\in \lbrace 0,1 \rbrace$, and define 
\[
h^{-1}(1^{u_{i}})=a_{0},2a_{1},2^{2}a_{2},\dots,2^{r}a_{r}.
\]
Now, when $\lambda_{i} \geq 3$, and for some odd $\ell_{i} \geq 1$, we have
\begin{displaymath}
h^{-1}(\lambda_{i}^{u_{i}})=\begin{cases} 
     \ell_{i}^{(u_{i}-1)s/2}, \lambda_{i} & \text{if  $u_{i}$ is odd and $\lambda_{i}=s\ell_{i}/2$, for all $\ell_{i}$,} \\
     (u_{i}-1)\lambda_{i}, \lambda_{i} & \text{if  $u_{i}$ is odd and  $\lambda_{i} \neq s \ell_{i}/2$, for all $\ell_{i}$,} \\
     \ell_{i}^{u_{i}s/2} & \text{if  $u_{i}$ is even and  $\lambda_{i}=s\ell_{i}/2$, for all $\ell_{i}$,} \\
     u_{i}\lambda_{i} & \text{if  $u_{i}$ is even and  $\lambda_{i}\neq s\ell_{i}/2$, for all $\ell_{i}$.} 
    \end{cases}
\end{displaymath}
As an illustration, for $(n,s)=(18,6)$ and the partition $\mu=(5^{2},3^{2},1^{2})$, we have
\[
h^{-1}(\mu) \rightarrow \Biggl\{ \begin{array}{l}
h^{-1}(5^{2}): u_{1}=2 \ \text{is even and} \ 5\neq 3\ell_{1} \rightarrow h^{-1}(5^{2})=10\\
h^{-1}(3^{2}) : u_{2}=2 \ \text{is even and} \ 3=3 \cdot 1 \rightarrow h^{-1}(3^{2})=1^{6}\\
h^{-1}(1^{2}): u_{3}=2^{1}\cdot 1 \rightarrow h^{-1}(1^{2})=2.
\end{array}
\]

\textbf{The bijection $\mathbb{M}_{s}(n)\Leftrightarrow \mathbb{D}_{s}(n)$ when $s\neq2^{p}$ for $p\geq 2$:}

Let $\lambda=(\lambda_{1}^{u_{1}}, \lambda_{2}^{u_{2}}, \dots, \lambda_{k}^{u_{k}})\in \mathbb{M}_{s}(n)$ and define the map $w:\mathbb{M}_{s}(n)\rightarrow \mathbb{D}_{s}(n)$ by $\lambda \rightarrow w(\lambda)=\bigcup_{i=1}^{k} w(\lambda_{i}^{u_{i}})$. If $\lambda_{i}$ is a non-multiple of $s/2$ with $u_{i} > 1$, then 
\begin{displaymath}
w(\lambda_{i}^{u_{i}})=\begin{cases} 
     (s\lambda_{i}/2)^{2u_{i}/s} & \text{if  $u_{i}\equiv 0$ (mod $s$),} \\
     (s\lambda_{i}/2)^{2(u_{i}-1)/s}, \lambda_{i} & \text{if  $u_{i}\equiv 1$ (mod $s$).}
    \end{cases}
\end{displaymath}
Else, 
\[
w(\lambda_{i}^{u_{i}})=\lambda_{i}^{u_{i}}.
\]
As an illustration, for $(n,s)=(18,6)$ and the partition $\lambda=(6, 2^{6})$, we have
\[
w(\lambda) \rightarrow \Biggl\{ \begin{array}{l}
w(6) \rightarrow 6 \equiv 0  \ ( \mathrm{mod} \ 3) \rightarrow w(6)=6\\
w(2^{6}) \rightarrow 2 \not\equiv 0  \ ( \mathrm{mod} \ 3) \ \text{and} \ u_{2}=6 \rightarrow w(2^{6}) = 6^{2}.
\end{array}
\]

Let $\lambda=(\lambda_{1}^{u_{1}}, \lambda_{2}^{u_{2}}, \dots, \lambda_{k}^{u_{k}})\in \mathbb{D}_{s}(n)$ and define the inverse map $w^{-1}:\mathbb{D}_{s}(n)\rightarrow \mathbb{M}_{s}(n)$ by $\lambda \rightarrow w(\lambda)=\bigcup_{i=1}^{k} w^{-1}(\lambda_{i}^{u_{i}})$. If $\lambda_{i} \equiv 0$ (mod $s/2$) and $u_{i}>1$, then $\lambda_{i}=s\ell_{i}/2$ where $\ell_{i}$ is a positive integer, and define
\begin{displaymath}
w^{-1}(\lambda_{i})=\begin{cases} 
    \ell_{i}^{su_{i}/2},  & \text{if  $u_{i}$ is even,} \\
    \ell_{i}^{s(u_{i}-1)/2}, \lambda_{i} & \text{if $u_{i}$ is odd.}
    \end{cases}
\end{displaymath}
Else, 
\[
w^{-1}(\lambda_{i}^{u_{i}})=\lambda_{i}^{u_{i}}.
\]
As an illustration, for $(n,s)=(18,6)$ and the partition $\mu=(6^{3})$, we have
\[
w^{-1}(\mu) \rightarrow 
w^{-1}(6^{3}) : 6=3\cdot 2 \rightarrow  w^{-1}(6^{3})=2^{6},6.
\]
\end{proof}

\begin{example} An illustration of the bijection $\mathbb{M}_{s}(n)\Leftrightarrow \mathbb{C}_{s}(n) \Leftrightarrow \mathbb{D}_{s}(n)$ is given for some of the partitions when $(s,n)=(4,10)$ in Table \ref{tab1}, and $(s,n)=(8,18)$ in Table \ref{tab2}.

\begin{table}[h!]
\caption{\label{tab1} The bijection of Theorem \ref{thmequivalence} for $n=10$, $s=4$.}
\begin{tabular}{@{}ccccc@{}}
\toprule
$\mathbb{M}_{4}(10)$   & $\Leftrightarrow$ & $\mathbb{D}_{4}(10)$   & $\Leftrightarrow$ & $\mathbb{C}_{4}(10)$ \\
\midrule
$(10)$        &                   & $(10)$        &                   & $(5^{2})$       \\ 
$(8,2)$       &                   & $(8,2)$       &                   & $(8,1^{2})$     \\ 
$(7,2,1)$     &                   & $(7,2,1)$     &                   & $(7,1^{3})$     \\ 
$(6,4)$       &                   & $(6,4)$       &                   & $(4,3^{2})$     \\ 
$(6,3,1)$     &                   & $(6,3,1)$     &                   & $(3^{3},1)$     \\
$(5,1^{5})$        &                   & $(5,2^{2},1)$        &                   & $(5,1^{5})$       \\ 
$(4,3,2,1)$   &                   & $(4,3,2,1)$   &                   & $(4,3,1^{3})$   \\
$(6,1^{4})$   &                   & $(6,2^{2})$   &                   & $(3^{2},1^{4})$ \\ 
$(4,2,1^{4})$ &                   & $(4,2^{3})$   &                   & $(4,1^{6})$     \\
$(3,2,1^{5})$ &                   & $(3,2^{3},1)$ &                   & $(3,1^{7})$     \\ 
$(2,1^{8})$   &                   & $(4^{2},2)$   &                   & $(1^{10})$      \\
$(2^{5})$     &                   & $(2^{5})$     &                   & $(4^{2},1^{2})$ \\
$(5, 1^{5})$     &                   & $(5, 2^{2}, 1)$     &                   & $(5, 1^{5})$ \\
$(9, 1)$     &                   & $(9, 1)$     &                   & $(9, 1)$ \\
$(5, 4, 1)$     &                   & $(5, 4, 1)$     &                   & $(5, 4, 1)$ \\
$(7, 3)$     &                   & $(7, 3)$     &                   & $(7, 3)$ \\
\botrule
\end{tabular}
\end{table}

\begin{table}[h!]
\caption{\label{tab2} The bijection of Theorem \ref{thmequivalence} for $n=18$, $s=8$.}
\begin{tabular}{@{}ccccc@{}}
\toprule
$\mathbb{M}_{8}(18)$   & $\Leftrightarrow$ & $\mathbb{D}_{8}(18)$   & $\Leftrightarrow$ & $\mathbb{C}_{8}(18)$ \\
\midrule
$(9,1^{9})$     &                   & $(9,4^{2},1)$   &                   & $(9,1^{9})$      \\ 
$(2^{9})$       &                   & $(4^{4},2)$     &                   & $(8^{2},1^{2})$  \\ 
$(8,2,1^{8})$   &                   & $(8,4^{2},2)$   &                   & $(8,1^{10})$     \\ 
$(7,3,1^{8})$   &                   & $(7,4^{2},3)$   &                   & $(7,3,1^{8})$    \\ 
$(7,2,1^{9})$   &                   & $(7,4^{2},2,1)$ &                   & $(7,1^{11})$     \\ 
$(10,1^{8})$    &                   & $(10,4^{2})$    &                   & $(5^{2},1^{8})$  \\ 
$(5,3,2,1^{8})$ &                   & $(5,4^{2},3,2)$ &                   & $(5,3,1^{10})$   \\ 
$(5,4,1^{9})$   &                   & $(5,4^{3},1)$   &                   & $(5,1^{13})$     \\ 
$(6,3,1^{9})$   &                   & $(6,4^{2},3,1)$ &                   & $(3^{3},1^{9})$  \\ 
$(6,4,1^{8})$   &                   & $(6,4^{3})$     &                   & $(3^{2},1^{12})$ \\ 
$(4,3,2,1^{9})$ &                   & $(4^{3},3,2,1)$ &                   & $(3,1^{15})$     \\ 
$(2,1^{16})$    &                   & $(8^{2},2)$     &                   & $(1^{18})$       \\ 
\botrule
\end{tabular}
\end{table}
\end{example}

\begin{example} An illustration of the bijection $\mathbb{M}_{s}(n)\Leftrightarrow \mathbb{C}_{s}(n) \Leftrightarrow \mathbb{D}_{s}(n)$ is given for some of the partitions when $(s,n)=(6,18)$ in Table \ref{tab3}, and $(s,n)=(10,15)$ in Table \ref{tab4}.

\begin{table}[h!]
\caption{\label{tab3} The bijection of Theorem \ref{thmequivalence} for $n=18$, $s=6$.}
\begin{tabular}{@{}ccccc@{}}
\toprule
$\mathbb{M}_{6}(18)$     & $\Leftrightarrow$ & $\mathbb{D}_{6}(18)$       & $\Leftrightarrow$ & $\mathbb{C}_{6}(18)$ \\
\midrule
$(12,1^{6})$    &                   & $(12,3^{2})$      &                   & $(12,3^{2})$       \\ 
$(10,2,1^{6})$  &                   & $(10,3^{2},2)$    &                   & $(5^{2},3^{2},1^{2})$ \\ 
$(9,3,1^{6})$   &                   & $(9,3^{3})$       &                   & $(9,3^{3})$        \\ 
$(8,4,1^{6})$   &                   & $(8,4,3^{2})$     &                   & $(3^{2},1^{12})$   \\ 
$(5,4,3,1^{6})$ &                   & $(5,4,3^{3})$     &                   & $(5,3^{3},1^{4})$  \\ 
$(11,1^{7})$    &                   & $(11,3^{2},1)$    &                   & $(11,3^{2},1)$     \\ 
$(8,3,1^{7})$   &                   & $(8,3^{3},1)$     &                   & $(3^{3},1^{9})$    \\ 
$(5,4,2,1^{7})$ &                   & $(5,4,3^{2},2,1)$ &                   & $(5,3^{2},1^{7})$  \\ 
$(6,1^{12})$    &                   & $(6,3^{4})$       &                   & $(6,3^{4})$        \\ 
$(5,1^{13})$    &                   & $(5,3^{4},1)$     &                   & $(5,3^{4},1)$      \\ 
$(4,2,1^{12})$  &                   & $(4,3^{4},2,1)$   &                   & $(3^{4},1^{6})$    \\ 
$(1^{18})$      &                   & $(3^{6})$         &                   & $(3^{6})$          \\ 
$(6,2^{6})$     &                   & $(6^{3})$         &                   & $(6^{3})$          \\ 
$(5,2^{6},1)$   &                   & $(6^{2},5,1)$     &                   & $(6^{2},5,1)$      \\ 
$(4,2^{7})$     &                   & $(6^{2},4,2)$     &                   & $(6^{2},1^{6})$    \\ 
$(3,2^{7},1)$   &                   & $(6^{2},3,2,1)$   &                   & $(6^{2},3,1^{3})$  \\ 
$(3^{6})$       &                   & $(9^{2})$         &                   & $(9^{2})$          \\
$(16,2)$        &                   & $(16,2)$          &                   & $(1^{18})$         \\ 
$(14,4)$        &                   & $(14,4)$          &                   & $(7^{2},1^{4})$    \\ 
\botrule
\end{tabular}
\end{table}

\begin{table}[h!]
\caption{\label{tab4} The bijection of Theorem \ref{thmequivalence} for $n=15$, $s=10$.}
\begin{tabular}{@{}ccccc@{}}
\toprule
$\mathbb{M}_{10}(15)$     & $\Leftrightarrow$ & $\mathbb{D}_{10}(15)$       & $\Leftrightarrow$ & $\mathbb{C}_{10}(15)$  \\
\midrule
$(14,1)$    &                   & $(14,1)$       &                   & $(7^{2},1)$       \\ 
$(12,3)$  &                   & $(12,3)$    &                   & $(3^{5})$ \\ 
$(12,2,1)$   &                   & $(12,2,1)$       &                   & $(3^{4},1^{3})$        \\ 
$(8,6,1)$    &                   & $(8,6,1)$     &                   & $(3^{2},1^{9})$     \\ 
$(8,5,2)$   &                   & $(8,5,2)$      &                   & $(5,1^{10})$    \\ 
$(8,4,2,1)$    &                   & $(8,4,2,1)$       &                   & $(1^{15})$        \\ 
$(7,6,2)$    &                   & $(7,6,2)$     &                   & $(7,3^{2},1^{2})$      \\ 
$(6,5,4)$      &                   & $(6,5,4)$         &                   & $(5,3^{2},1^{4})$          \\ 
$(6,5,3,1)$     &                   & $(6,5,3,1)$         &                   & $(5,3^{3},1)$          \\ 
$(6,4,3,2)$   &                   & $(6,4,3,2)$     &                   & $(3^{3},1^{6})$      \\ 
$(3,2,1^{10})$   &                   & $(5^{2},3,2)$   &                   & $(5^{2},3,1^{2})$  \\ \
$(4,1^{11})$       &                   & $(5^{2},4,1)$         &                   & $(5^{2},1^{5})$          \\ 
$(5, 4, 3, 2, 1)$       &                   & $(5, 4, 3, 2, 1)$        &                   & $(5, 3, 1^{7})$          \\ 
$(7, 4, 3, 1)$       &                   & $(7, 4, 3, 1)$        &                   & $(7, 3, 1^{5})$          \\ 
\botrule
\end{tabular}
\end{table}
\end{example}

It is well known that recurrence relations play a crucial role in computing the number of partitions of every positive integer $n$ by using compact and elegant formulas
for the number of partitions of large positive integers. Therefore, we establish recurrence relations for the three classes. Let $M_{s}(n,k)$ denote the number of $s$-modular partitions of $n$ into $k$ parts. By convention we define $M_{s}(0,0)=1$.

\begin{theorem} For every positive integers $n,k \geq 1$,
\[
M_{s}(n,k)=\sum\limits_{\substack{ \ell=0 \\ \ell\equiv 0,1\ (\mathrm{mod}\ s)} }^{k} M_{s}(n-k,k-\ell).
\]
\end{theorem}
 
\begin{proof}
In every partition $\lambda\in \mathbb{M}_{s}(n,k)$, the part $1$ has a multiplicity of $\ell$ for some $\ell \equiv 0,1$ (mod $s$). Removing the $1$'s from $\lambda$ and subtracting $1$ from its other parts yields a partition of $n-k$ into $k-\ell$ parts. If the partition $\lambda$ consists only of parts of size $1$, then we obtain a partition of $0$, that is, $M_{s}(0,0)=1$.
\end{proof}
 The following array contains some values of $M_{4}(n,k)$ for $n \leq 20$ and $k \leq 15$.
 
 \begin{table}[h!]
\caption{Some values of $M_{4}(n,k)$ for $n \leq 20$ and $k \leq 15$.}
\begin{tabular}{@{}cccccccccccccccc@{}}
\toprule
$n$ & &  & & & & & & & & & & & & &  \\
\toprule
1 &1&&&&&&\\
2 &1&&&&&&\\
3 &1&1&&&&&\\
4 &1&1&&1&&&\\
5 &1&2&&&1&&\\
6 &1&2&1&&1&\\
7 &1&3&1&&1&1&\\
8 &1&3&2&1&1&1&&1&\\
9 &1&4&3&&2&2&&&1&\\
10 &1&4&4&1&2&2&1&&1&\\
11 &1&5&5&1&2&4&1&&1&1&\\
12 &1&5&7&3&2&4&2&1&1&1&&1&\\
13 &1&6&8&3&3&6&3&&2&2&&&1&\\
14 &1&6&10&5&3&6&5&1&2&2&1&&1&\\
15 &1&7&12&6&4&9&6&1&2&4&1&&1&1&\\
16 &1&7&14&10&4&9&9&&2&4&2&1&1&1&\\
17 &1&8&16&11&5&13&11&3&4&8&1&&2&2&\\
18 &1&8&19&15&7&12&15&6&4&6&5&1&2&2&1\\
19 &1&9&21&18&9&16&18&7&5&10&6&1&2&4&1\\
20 &1&9&24&24&11&16&23&13&5&10&9&4&2&4&2\\
\toprule
$k$ & 1 & 2 &3&4&5&6&7&8&9&10&11&12&13&14&15\\
\botrule
\end{tabular}
\end{table}

\begin{example} For $(s,n,k)=(4,20,8)$, we have
\begin{align*} 
M_{4}(20,8)&=M_{4}(12,8)+M_{4}(12,7)+M_{4}(12,4)+M_{4}(12,3)\\
&=1+2+3+7\\
&=13.
\end{align*}

\begin{table}[h!]
\begin{tabular}{|c|c|c|c|c|}
\hline
$\mathbb{M}_{4}(20,8)$     & $\mathbb{M}_{4}(12,8)$   & $\mathbb{M}_{4}(12,7)$ & $\mathbb{M}_{4}(12,4)$ & $\mathbb{M}_{4}(12,3)$ \\ \hline
$(3^{4},2^{4})$   & $(2^{4},1^{4})$ &               &               &               \\ \hline
$(6,3,2^{5},1)$   &                 & $(5,2,1^{5})$ &               &               \\ \hline
$(6,3,2^{5},1)$   &                 & $(4,3,1^{5})$ &               &               \\ \hline
$(7,4,3,2,1^{4})$ &                 &               & $(6,3,2,1)$   &               \\ \hline
$(6,5,3,2,1^{4})$ &                 &               & $(5,4,2,1)$   &               \\ \hline
$(4^{4},1^{4})$   &                 &               & $(3^{4})$     &               \\ \hline
$(10,3,2,1^{5})$  &                 &               &               & $(9,2,1)$     \\ \hline
$(9,4,2,1^{5})$   &                 &               &               & $(8,3,1)$     \\ \hline
$(8,5,2,1^{5})$   &                 &               &               & $(7,4,1)$     \\ \hline
$(8,4,3,1^{5})$   &                 &               &               & $(7,3,2)$     \\ \hline
$(7,6,2,1^{5})$   &                 &               &               & $(6,5,1)$     \\ \hline
$(7,5,3,1^{5})$   &                 &               &               & $(6,4,2)$     \\ \hline
$(6,5,4,1^{5})$   &                 &               &               & $(5,4,3)$     \\ \hline
\end{tabular}
\caption{An illustration of the construction of the set $\mathbb{M}_{4}(20,8)$.}
\end{table}
\end{example}

Let $C_{s}(n,k)$ denote the number of $s$-congruent partitions of $n$ into $k$ parts such that $C_{s}(n,k):=\mid \mathbb{C}_{s}(n,k) \mid$. First, we introduce some definitions that are needed for the proof of the next recurrence relation. Let $N(s)$ be the set consisting of $s$ and all the odd positive integers less than $s$, and let $N'(\ell)$ be the set consisting of all the odd positive integers less than $\ell \in N(s)$. Denote by $\mathbb{C}^{\ell}_{s}(n,k)$ the set of all the partitions $\lambda \in \mathbb{C}_{s}(n,k)$ with at least one part of size $\ell \in N(s)$ as the smallest part in the partition, and by $\mathbb{C}^{s+1}_{s}(n,k)$ the set of all the partitions $\lambda \in \mathbb{C}_{s}(n,k)$ into parts of sizes greater than $s$. Therefore, we have the following dissection of $\mathbb{C}_{s}(n,k)$ into $\mid N(s) \mid +1$ disjoint subsets
\begin{equation}
\mathbb{C}_{s}(n,k)=\mathbb{C}_{s}^{1}(n,k)\cup \mathbb{C}_{s}^{3}(n,k)\cup\dots\cup\mathbb{C}_{s}^{s}(n,k)\cup\mathbb{C}_{s}^{s+1}(n,k). \label{dissection}
\end{equation}
Consider the subset $\mathbb{C}_{s}(n-\ell,k-1)\setminus \cup_{i\in N'(\ell)} \mathbb{C}_{s}^{i}(n-\ell,k-1)$ which contains all the partitions of $n-\ell$ into $k-1$ parts of sizes $\geq \ell$, where $\ell \in N(s)$. To obtain partitions of $n-\ell$ into $k-1$ parts of sizes $\geq \ell$, we exclude those with at least one part of size $i\in N'(\ell)$ as the smallest part from $\mathbb{C}_{s}(n-\ell,k-1)$. Conversely, by adding one part of size $\ell \in N(s)$ to each partition $\lambda \in \mathbb{C}_{s}(n-\ell,k-1)\setminus \cup_{i\in N'(\ell)} \mathbb{C}_{s}^{i}(n-\ell,k-1)$, we arrive at $\mathbb{C}^{\ell}_{s}(n,k)$. Therefore, for all $\ell \in N(s)$ we deduce that
\[
\mid \mathbb{C}_{s}^{\ell}(n,k) \mid = \mid \mathbb{C}_{s}(n-\ell,k-1)\setminus \cup_{i\in N'(\ell)} \mathbb{C}_{s}^{i}(n-\ell,k-1) \mid.
\]
Similarly, by adding $s$ to each part of $\lambda\in \mathbb{C}_{s}(n-sk,k)$, we get $\mathbb{C}_{s}^{s+1}(n,k)$. Therefore, we find that
\[
\mid \mathbb{C}_{s}^{s+1}(n,k) \mid= \mid \mathbb{C}_{s}(n-sk,k) \mid.
\]
Then, we deduce the following relations:
\[
C_{s}^{\ell}(n,k)=C_{s}(n-\ell,k-1)-\sum_{i\in N'(\ell)}C_{s}^{i}(n-\ell,k-1) \ \text{and} \ C_{s}^{s+1}(n,k) = C_{s}(n-sk,k).
\] 
By convention we define $C_{s}(0,0)=1$ and $C_{s}^{1}(n,k)=C_{s}(n-1,k-1)$.

\begin{theorem} For every positive integers $n,k \geq 1$, we have
\[
C_{s}(n,k)=\sum\limits_{\ell\in N(s)} C_{s}^{\ell}(n,k)+C_{s}^{s+1}(n,k).
\]
Moreover, for $\ell \in N(s)$, we have
\[
C_{s}^{\ell}(n,k)=C_{s}(n-\ell,k-1)-\sum_{i\in N'(\ell)}C_{s}^{i}(n-\ell,k-1) \ \text{and} \ C_{s}^{s+1}(n,k)=C_{s}(n-sk,k).
\] 
\end{theorem}

\begin{proof} Let the sets $N(s)=\lbrace 1,3,5,\dots, (s-1), s\rbrace$ and $N'(\ell)=\lbrace 1,3,5,\dots, m\rbrace$ where $m<\ell \in N(s)$, and consider the dissection $\eqref{dissection}$ of $\mathbb{C}_{s}(n,k)$. Then, for evey partition $\lambda \in \mathbb{C}^{\ell}_{s}(n,k)$, remove one part of size $\ell$ to obtain a partition of $n-\ell$ into $k-1$ parts, that is, a partition $\lambda\in \mathbb{C}_{s}(n-\ell,k-1)\setminus \cup_{i\in N'(\ell)} \mathbb{C}_{s}^{i}(n-\ell,k-1)$. If $\lambda \in \mathbb{C}^{s+1}_{s}(n,k)$, subtract $s$ from each of the $k$ parts of $\lambda$ to obtain a partition of $n-sk$ into $k$ parts. Therefore, we have
\[
C_{s}^{\ell}(n,k)=C_{s}(n-\ell,k-1)-\sum_{i\in N'(\ell)}C_{s}^{i}(n-\ell,k-1) \ \text{and} \ C_{s}^{s+1}(n,k) = C_{s}(n-sk,k).
\]
\end{proof}

\begin{table}[h!]
\caption{Some values of $C_{4}(n,k)$ for $n \leq 20$ and $k \leq 15$.}
\begin{tabular}{@{}cccccccccccccccc@{}}
\toprule
$n$ & &  & & & & & & & & & & & & &  \\
\toprule
1 &1&\\
2 &&1&\\
3 &1&&1&\\
4 &1&1&&1&\\
5 &1&1&1&&1&\\
6 &&2&1&1&&1&\\
7 &1&1&2&1&1&&1&\\
8 &1&3&1&2&1&1&&1&\\
9 &1&2&4&1&2&1&1&&1&\\
10 &&3&3&4&1&2&1&1&&1&\\
11 &1&2&5&3&4&1&2&1&1&&1&\\
12 &1&4&4&6&3&4&1&2&1&1&&1&\\
13 &1&3&7&5&6&3&4&1&2&1&1&&1&\\
14 &&4&6&9&5&6&3&4&1&2&1&1&&1&\\
15 &1&3&9&8&10&5&6&3&4&2&2&1&1& &1\\
16 &1&6&7&13&9&10&5&6&3&4&1&2&1&1&\\
17 &1&4&12&11&15&9&10&5&6&3&4&1&2&1&1\\
18 &&5&10&18&13&16&9&10&5&6&3&4&1&2&1\\
19 &1&4&14&16&22&14&16&9&10&5&6&3&4&1&2\\
20 &1&7&12&23&21&24&14&16&9&10&5&6&3&4&1\\
\toprule
$k$ & 1 & 2 &3&4&5&6&7&8&9&10&11&12&13&14&15\\
\botrule
\end{tabular}
\end{table}
For brevity, let $\mathbb{K}_{\geq\ell}(n-\ell,k-1)=\mathbb{C}_{s}(n-\ell,k-1)\setminus \cup_{i\in N'(\ell)} \mathbb{C}_{s}^{i}(n-\ell,k-1)$.
\begin{example} For $(s,n,k)=(4,20,4)$, we have
\[
C_{4}(20,4)=C_{4}^{1}(20,4)+C_{4}^{3}(20,4)+C_{4}^{4}(20,4)+C_{4}^{5}(20,4)=14+6+2+1=23,
\]
such that
\begin{itemize}
\item $C_{4}^{1}(20,4)=C_{4}(19,3)=14$,
\item $C_{4}^{3}(20,4)=C_{4}(17,3)-C_{4}^{1}(17,3)=C_{4}(17,3)-C_{4}(16,2)=12-6=6$,
\item $C_{4}^{4}(20,4)=C_{4}(16,3)-C_{4}^{1}(16,3)-C_{4}^{3}(16,3)=C_{4}(16,3)-C_{4}(15,2)-( C_{4}(13,2)-C_{4}(12,1))=7-3-(3-1)=2$,
\item $C_{4}^{5}(20,4)=C_{4}(4,4)=1$.
\end{itemize}

\begin{table}[h!]
\begin{tabular}{|c|c|c|c|c|}
\hline
$\mathbb{C}_{4}(20,4)$     & $\mathbb{K}_{\geq 1}(19,3)$   & $\mathbb{K}_{\geq 3}(17,3)$ & $\mathbb{K}_{\geq 4}(16,3)$ & $\mathbb{C}_{4}(4,4)$ \\ \hline
$(17,1^{3})$   &       $(17,1^{2})$           &               &               &               \\ \hline
$(15,3,1^{2})$   &      $(15,3,1)$           &               &               &               \\ \hline
$(13,5,1^{2})$   &        $(13,5,1)$         &               &               &               \\ \hline
$(13,3^{2},1)$ &        $(13,3^{2})$         &               &               &               \\ \hline
$(12,4,3,1)$ &        $(12,4,3)$         &               &               &               \\ \hline
$(17,1^{3})$  &        $(17,1^{2})$         &               &               &               \\ \hline
$(11,7,1^{2})$  &     $(11,7,1)$            &               &               &     \\ \hline
$(11,5,3,1)$   &       $(11,5,3)$          &               &               &     \\ \hline
$(11,4^{2},1)$   &        $(11,4^{2})$          &               &               &     \\ \hline
$(9^{2},1^{2})$   &        $(9^{2},1)$         &               &               &     \\ \hline
$(9,7,3,1)$   &       $(9,7,3)$          &               &               &      \\ \hline
$(9,5^{2},1)$   &     $(9,5^{2})$            &               &               &     \\ \hline
$(8^{2},3,1)$   &      $(8^{2},3)$           &               &               &      \\ \hline
$(8,7,4,1)$ &       $(8,7,4)$          &               &               &               \\ \hline
$(7^{2},5,1)$   &        $(7^{2},5)$         &               &               &               \\ \hline
$(11,3^{3})$  &                &       $(11,3^{2})$        &               &     \\ \hline
$(9,5,3^{2})$   &                 &     $(9,5,3)$          &               &      \\ \hline
$(9,4^{2},3)$   &                 &     $(9,4^{2})$          &               &      \\ \hline
$(8,5,4,3)$   &                 &      $(8,5,4)$          &               &     \\ \hline
$(7^{2},3^{2})$   &                 &    $(7^{2},3)$           &               &     \\ \hline
$(7,5^{2},3)$   &                 &        $(7,5^{2})$       &               &     \\ \hline
$(8,4^{3})$   &                &               &     $(8,4^{2})$           &     \\ \hline
$(7,5,4^{2})$   &                 &               &     $(7,5,4)$          &     \\ \hline
$(5^{4})$   &                 &               &               &   $(1^{4})$  \\ \hline
\end{tabular}
\caption{An illustration of the construction of $\mathbb{C}_{4}(20,4)$.}
\end{table}
\end{example}
In the next theorem we use the union operation on partitions, where figure \ref{youngdiagramoftheunion} shows the definition of the operation using Ferrers diagram of the partitions $\lambda=(4)$ and $\beta=(3, 2)$.
\begin{figure} [h!]
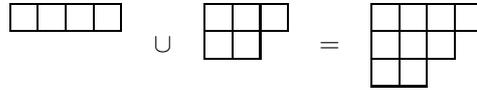

\begin{center}
\ytableausetup
{mathmode, boxframe=normal, boxsize=1em}
\begin{ytableau}

 & & & &  \none[] &  \none[] & \none[]  & & & &  \none[] & \none[] & \none[] & & & &\\
\none[] &  \none[] &  \none[] &  \none[] &  \none[] & \none[\cup] & \none[]  & & &  \none[] &  \none[] & \none[=] & \none[] & & &\\
\none[] &  \none[] &  \none[] &  \none[] &  \none[] &  \none[] &  \none[] &  \none[] &  \none[] &  \none[] &  \none[] &  \none[] &   \none[] & & \\
\end{ytableau}
\end{center}
\caption{Ferrers diagram of $\lambda \cup \beta=(4, 3, 2)$. \label{youngdiagramoftheunion}}
\end{figure}

Let $\mathbb{D}_{s}(n,k)$ denote the set of $s$-duplicate partitions of $n$ into $k$ parts such that $D_{s}(n,k):=\mid \mathbb{D}_{s}(n,k) \mid$. By convention we define $D_{s}(0,0)=1$. Let $\mathbb{O}(s)$ denote the set of all partitions into distinct parts less than or equal to $s/2-1$. Let $\mathbb{A}(s)=\mathbb{O}(s) \cup \lbrace \emptyset \rbrace$, where by convention we define the empty set as the empty partition and $\mid \mathbb{A}(s) \mid = 2^{s/2-1}$. Now, consider the partition $\alpha^{j} \in \mathbb{A}(s)$, where $0\leq \mid \alpha^{j} \mid \leq \frac{(s/2-1)(s/2)}{2}$ and $0 \leq \ell(\alpha^{j}) \leq s/2-1$ for $1\leq j \leq 2^{s/2-1}$. By convention we consider the length of the partition of $\emptyset$ to be $0$ along with its value. Moreover, consider the partition $\gamma^{j}$ such that $\lambda= \gamma^{j} \cup \alpha^{j}$ for some $\lambda \in \mathbb{D}_{s}(n,k)$. There are two kinds of partitions $\lambda \in \mathbb{D}_{s}(n,k)$. The first kind consists of partitions with at least one part of size $s/2$, while the second one consists of partitions into exactly $k-\ell(\alpha^{j})$ parts of sizes greater than $s/2$ and $\ell(\alpha^{j})$ parts of sizes $\lambda_{i} \leq s/2-1$. That is, the partitions of the second kind are of the form $\lambda= \gamma^{j} \cup \alpha^{j}$ for $1\leq j \leq 2^{s/2-1}$. We denote by $\mathbb{D}_{s}(n,k;s/2)$ the subset of the first kind and by $\mathbb{D}_{s}(n,k;\alpha^{j})$ the subset of the second. The dissection of $\mathbb{D}_{s}(n,k)$ into $2^{s/2-1}+1$ disjoint subsets is given by
\begin{equation}
\mathbb{D}_{s}(n,k)= \mathbb{D}_{s}(n,k;s/2) \cup \mathbb{D}_{s}(n,k;\alpha^{1}) \cup \cdots \cup \mathbb{D}_{s}(n,k;\alpha^{2^{s/2-1}}). \label{dissection1}
\end{equation}

\begin{theorem} \label{thmduplicate} For every positive integers $n,k \geq 1$,
\[
D_{s}(n,k)=D_{s}(n-s/2,k-1)+\sum_{\alpha^{j} \in \mathbb{A}(s)}D_{s} \Big( n-\frac{s(k-\ell(\alpha^{j}))}{2}-\mid \alpha^{j} \mid,k-\ell(\alpha^{j}) \Big).
\]
\end{theorem}

\begin{proof} Consider the dissection $\eqref{dissection1}$ of $\mathbb{D}_{s}(n,k)$. For every partition $\lambda\in\mathbb{D}_{s}(n,k;s/2)$, remove one part of size $s/2$ to obtain a partition of $n-s/2$ into $k-1$ parts. On the other hand, for $1\leq j \leq 2^{s/2-1}$ and for every $\lambda\in\mathbb{D}_{s}(n,k;\alpha^{j})$, subtract $s/2$ from the $k-\ell(\alpha^{j})$ parts greater than $s/2$ and remove $\alpha^{j}$ to obtain a partition of $n-s(k-\ell(\alpha^{j}))/2-\mid \alpha^{j} \mid$ into $k-\ell(\alpha^{j})$ parts.
\end{proof}

\begin{table}[h!]
\caption{Some values of $D_{4}(n,k)$ for $n\leq 14$ and $k\leq 7$.}
\begin{tabular}{@{}cccccccccccccccccc@{}}
\toprule
& & & & & & & & & & & & & &  & &\\
\toprule
& & & & & 1 &1&\\
& & & & &2 &1&\\
& & & & &3 &1&1&\\
& & & & &4 &1&2&\\
& & & & &5 &1&2&1&\\
& & & & &6 &1&2&2&\\
& & & & &7 &1&3&2&1&\\
& & & & &8 &1&4&3&2&\\
& & & & &9 &1&4&5&2&1&\\
& & & & &10 &1&4&6&3&2&\\
& & & & &11 &1&5&7&5&2&1&\\
& & & & &12 &1&6&9&7&3&2&\\
& & & & &13 &1&6&11&9&5&2&1&\\
& & & & &14 &1&6&13&11&7&3&2&\\
\toprule
& & & & & $k$ &1&2&3&4&5&6&7& & & & &\\
\botrule
\end{tabular}
\end{table}

\begin{example} \label{example4duplicate} For $s=4$, we have the set $\mathbb{O}(4)=\lbrace (1) \rbrace$. Therefore, $\mathbb{A}(4)=\lbrace \emptyset, (1) \rbrace$ and the recurrence relation is given by
\begin{equation}
D_{4}(n,k)=D_{4}(n-2,k-1)+D_{4}(n-2k,k)+D_{4}(n-2(k-1)-1,k-1). \label{4duprec}
\end{equation}
For $s=8$, we have the set $\mathbb{O}(8)=\lbrace (1), (2), (3), (2,1), (3,1), (3,2), (3,2,1)\rbrace$. Therefore, $\mathbb{A}(8)=\lbrace \emptyset, (1), (2), (3), (2,1), (3,1), (3,2), (3,2,1) \rbrace$ and the recurrence relation is given by
\begin{multline*}
D_{8}(n,k)=D_{8}(n-4k,k-1)+D_{8}(n-4k,k)+D_{8}(n-4(k-1)-1,k-1)\\
+D_{8}(n-4(k-1)-2,k-1)\\
+D_{8}(n-4(k-1)-3,k-1)\\
+D_{8}(n-4(k-2)-3,k-2)\\
+D_{8}(n-4(k-2)-4,k-2)\\
+D_{8}(n-4(k-2)-5,k-2)\\
+D_{8}(n-4(k-3)-6,k-3).\\
\end{multline*}
\end{example}

\begin{example} For $(s,n,k)=(4,13,3)$, we have
\begin{align*} 
D_{4}(13,3)&=D_{4}(11,2)+D_{4}(7,3)+D_{4}(8,2)\\
&=5+2+4\\
&=11.
\end{align*}
\begin{table}[h!]
\begin{tabular}{|c|c|c|c|c|}
\hline
$\mathbb{D}_{4}(13,3)$     & $\mathbb{D}_{4}(11,2)$   & $\mathbb{D}_{4}(7,3)$ & $\mathbb{D}_{4}(8,2)$  \\ \hline
$(10,2,1)$   & $(10,1)$ &               &                           \\ \hline
$(9,2^{2})$   &   $(9,2)$              &  &                             \\ \hline
$(8,3,2)$   &      $(8,3)$           &  &                           \\ \hline
$(7,4,2)$ &     $(7,4)$            &               &                \\ \hline
$(6,5,2)$ &      $(6,5)$           &               &                 \\ \hline
$(6,4,3)$   &                 &       $(4,2,1)$        &                   \\ \hline
$(5,4^{2})$  &                 &    $(3,2^{2})$           &                    \\ \hline
$(9,3,1)$  &                 &            &                   $(7,1)$ \\ \hline
$(7,5,1)$   &                 &               &                  $(5,3)$ \\ \hline
$(6^{2},1)$   &                 &               &                  $(4^{2})$ \\ \hline
$(8,4,1)$   &                 &               &                   $(6,2)$ \\ \hline
\end{tabular}
\caption{An illustration of the construction of the set $\mathbb{D}_{4}(13,3)$.}
\end{table}
\end{example}

Merca \cite{r18} proved that the partition function $\myped(n)$, which enumerates the number of partitions of $n$ wherein even parts are distinct and odd parts are unrestricted, satisfies 
\[
\myped(n)=\sum_{k\geq 0}C_{4}(n-2T_{k}),
\]
where $T_{k}=k(k+1)/2$ is the $k$th triangular number. In the spirit of Euler's recurrence for the unrestricted partition function $p(n)$, we use a simple observation to prove the following recurrence for $C_{4}(n)$.
\begin{theorem} For every positive integer $n\geq 1$,
\[
C_{4}(n)=\sum\limits_{k\geq 1}(-1)^{k(k+1)/2+1}C_{4}(n-k(k+1)/2).
\]
\end{theorem}

\begin{proof}
We have that 
\begin{equation}
\psi(q)= \dfrac{(q^{2};q^{2})^{2}_{\infty}}{(q;q)_{\infty}}
= \sum\limits_{k\geq 0}q^{k(k+1)/2}. \label{psi}
\end{equation}
Replacing $q$ by $-q$ in $\eqref{psi}$ and using $(-q;-q)_{\infty}=\frac{(q^{2};q^{2})^{3}_{\infty}}{(q;q)_{\infty}(q^{4};q^{4})_{\infty}}$, we obtain
\[
\psi(-q)= \dfrac{(q;q)_{\infty}(q^{4};q^{4})_{\infty}}{(q^{2};q^{2})_{\infty}}\\
= \sum\limits_{k\geq 0}(-q)^{k(k+1)/2}.
\]
Then $(\sum\limits_{k\geq 0}(-q)^{k(k+1)/2})(\sum\limits_{n\geq 0}C_{4}(n)q^{n})=1$, from which our result follows.
\end{proof}

\begin{table}[h!]
\begin{tabular}{|c|c|c|c|c|c|c|c|c|c|}
\hline
$C_{4}(0)$ & $1$ & $C_{4}(5)$ & $4$  & $C_{4}(10)$ & $16$ & $C_{4}(15)$ & $55$  & $C_{4}(20)$ & $161$ \\ \hline
$C_{4}(1)$ & $1$ & $C_{4}(6)$ & $5$  & $C_{4}(11)$ & $21$ & $C_{4}(16)$ & $70$  & $C_{4}(21)$ & $196$ \\ \hline
$C_{4}(2)$ & $1$ & $C_{4}(7)$ & $7$  & $C_{4}(12)$ & $28$ & $C_{4}(17)$ & $86$  & $C_{4}(22)$ & $236$ \\ \hline
$C_{4}(3)$ & $2$ & $C_{4}(8)$ & $10$ & $C_{4}(13)$ & $35$ & $C_{4}(18)$ & $105$ & $C_{4}(23)$ & $287$ \\ \hline
$C_{4}(4)$ & $3$ & $C_{4}(9)$ & $13$ & $C_{4}(14)$ & $43$ & $C_{4}(19)$ & $130$ & $C_{4}(24)$ & $350$ \\ \hline
\end{tabular}
\caption{The first $24$ values of the function $C_{4}(n)$.}
\end{table}

\begin{example} For $n=21$, we obtain
\begin{eqnarray*}
C_{4}(21) &=& C_{4}(20)+C_{4}(18)-C_{4}(15)-C_{4}(11)+C_{4}(6)+C_{4}(0)\\
&=& 161+105-55-21+5+1 \\
&=& 196,
\end{eqnarray*}
and for $n=24$, we get
\begin{eqnarray*}
C_{4}(24) &=& C_{4}(23)+C_{4}(21)-C_{4}(18)-C_{4}(14)+C_{4}(9)+C_{4}(3)\\
&=& 287+196-105-43+13+2 \\
&=& 350.
\end{eqnarray*}
\end{example}

Following the fact that $\mypod(n)$ is a special case of the $s$-duplicate partitions, we give a generalization for the series expansion \eqref{eqAlladi} using the $s/2$-modular Ferrers diagrams. That is, we wish to obtain an expansion for the product generating function
\[
\sum\limits_{n,r,l\geq 0} D_{s}(n, r, l)z^{r}b^{l}q^{n}=  \prod\limits_{n\geq 1}\dfrac{(1+zbq^{n})}{(1+zbq^{sn/2})(1-zq^{sn/2})},
\]
where the powers of $b$ and $z$ encode the number of distinct parts and the total number of parts respectively.

\begin{theorem} \label{podgen} We have
\begin{multline*}
\sum_{n, r, l\geq 0}D_{s}(n, r, l)z^{r}b^{l}q^{n}=1+\sum_{k\geq 1} A_{k}(q) \times \\
\biggl\{ z^{k}q^{sk^{2}/2}(-zbq^{s(k-1)/2+1};q)_{s/2-1} (-bq^{s(k-1)/2+1};q)_{s/2-1}\\
+(1-zq^{sk/2})(1-q^{sk/2}) \sum_{i= 1}^{s/2-1} bz^{k}q^{sk^{2}/2-i}\biggr\},
\end{multline*}
where 
\[
A_{k}(q)= \dfrac{(-zbq,-zbq^{2}, \dots,-zbq^{s/2-1}; q^{s/2})_{k-1}(-bq,-bq^{2}, \dots,-bq^{s/2-1}; q^{s/2})_{k-1}}{(zq^{s/2};q^{s/2})_{k}(q^{s/2};q^{s/2})_{k}}.
\]
\end{theorem}

\begin{proof} Consider the $k\times k$ Durfee square $D$ of the $s/2$-modular Ferrers diagram of an $s$-duplicate partition $\lambda$ of $n$, which is the largest square of boxes starting from the top-left corner of the diagram. This dissection of $\lambda$ results a partition $\lambda^{r}$ to the right of the square $D$ and a partition $\lambda^{b}$ below the square $D$. With regard to Durfee squares of $s/2$-modular diagrams, there are $s/2$ cases to consider. The first case is when the bottom-right box of the square $D$ has an entry $s/2$, while in the other $s/2-1$ cases, the bottom-right box of $D$ has an entry $i$ with $1\leq i \leq s/2-1$. In the first case, $\lambda^{r}$ is a partition into at most $k$ parts and $\lambda^{b}$ is a partition into parts $\leq sk/2$, and in the rest of the cases, $\lambda^{r}$ is a partition into at most $k-1$ parts and $\lambda^{b}$ is a partition into parts $\leq s(k-1)/2$. Note that both $\lambda^{r}$ and $\lambda^{b}$ are also $s$-duplicate partitions.

\begin{figure} [h!]
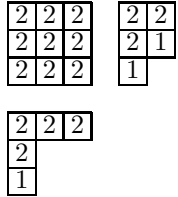

\begin{center}
\ytableausetup
{mathmode, boxframe=normal, boxsize=1em}
\begin{ytableau}
 2 & 2 & 2 & \none[] & 2 & 2\\
 2 & 2 & 2 & \none[] & 2 & 1\\
 2 & 2 & 2 & \none[] & 1\\
 \none[] & \none[] & \none[] & \none[] & \none[] \\
 2 & 2 & 2 \\
 2 \\
 1 \\
\end{ytableau}
\caption{An illustration of the first case for $s=4$, $k=3$, and $\lambda=(10,9,7,6,2,1)$.}
\end{center}
\end{figure}

\begin{figure} [h!]
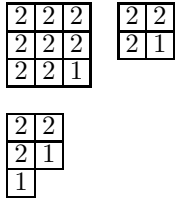

\begin{center}
\ytableausetup
{mathmode, boxframe=normal, boxsize=1em}
\begin{ytableau}
 2 & 2 & 2 & \none[] & 2 & 2\\
 2 & 2 & 2 & \none[] & 2 & 1\\
 2 & 2 & 1  \\
 \none[] & \none[] & \none[] & \none[] & \none[] \\
 2 & 2  \\
 2 & 1 \\
 1 \\
\end{ytableau}
\caption{An illustration of the other $s/2-1$ cases for $s=4$, $k=3$, and $\lambda=(10,9,5,4,3,1)$.}
\end{center}
\end{figure}

\noindent\textbf{The generating function for the number of partitions in the first case:} Let $\lambda \in \mathbb{D}(n,k)$ whose $s/2$-modular Ferrers diagram has a $k\times k$ Durfee square $D$ with all boxes filled with $s/2$. The sum of the entries in the boxes of $D$ is $sk^{2}/2$ and so we have the term
\[
z^{k}q^{sk^{2}/2}.
\]
On the other hand, the generating function for the partitions $\lambda^{b}$ is given by
\begin{equation}
\dfrac{(-zbq, -zbq^{2}, \dots,-zbq^{s/2-1}; q^{s/2})_{k}}{(zq^{s/2};q^{s/2})_{k}}. \label{eqb}
\end{equation}
In computing the generating function of $\lambda^{r}$, we do not need to keep track of the total number of parts. Thus the parameter $z$ will be absent in this generating function. Instead, we are only interested in the number of distinct parts in $\lambda^{r}$. Then, by using the conjugation of the modular diagram, we obtain the following generating function
\begin{equation}
\dfrac{(-bq, -bq^{2}, \dots,-bq^{s/2-1}; q^{s/2})_{k}}{(q^{s/2};q^{s/2})_{k}}. \label{eqr}
\end{equation}
Thus, the generating function for the partitions in the first case is
\[
z^{k}q^{sk^{2}/2}\dfrac{(-zbq, -zbq^{2}, \dots,-zbq^{s/2-1}; q^{s/2})_{k}(-bq, -bq^{2}, \dots,-bq^{s/2-1}; q^{s/2})_{k}}{(zq^{s/2};q^{s/2})_{k}(q^{s/2};q^{s/2})_{k}}.
\]

\noindent\textbf{The generating function for the number of partitions in the other $s/2-1$ cases:} The rest of the $s/2-1$ cases of partitions whose $s/2$-modular $k\times k$ Durfee square $D$ has a sum of entries equals to $s(k^{2}-1)/2+i$ where $i\in \lbrace 1,2, \dots ,s/2-1 \rbrace$, such that $i$ is the entry in the bottom right box of $D$. Again, the partitions $\lambda^{b}$ and $\lambda^{r}$ are the same as in the previous first case, except that the largest part of $\lambda^{b}$ is $\leq s(k-1)/2$ and the total number of parts of $\lambda^{r}$ is $\leq k-1$. Thus, the generating function for such cases would be
\[
\sum\limits_{i=1}^{s/2-1} bz^{k}q^{sk^{2}/2-i}\dfrac{(-zbq, \dots,-zbq^{s/2-1}; q^{s/2})_{k-1}(-bq, \dots,-bq^{s/2-1}; q^{s/2})_{k-1}}{(zq^{s/2};q^{s/2})_{k-1}(q^{s/2};q^{s/2})_{k-1}}.
\]
For the products occurring in \eqref{eqb} and \eqref{eqr}, we have
\begin{eqnarray*}
(-zbq, \dots,-zbq^{\frac{s}{2}-1}; q^{\frac{s}{2}})_{k} &=& (-zbq, \dots,-zbq^{\frac{s}{2}-1}; q^{\frac{s}{2}})_{k-1} \prod_{t=1}^{s/2-1}(1+zbq^{t}q^{s(k-1)/2}) \\  &=& (-zbq, \dots,-zbq^{\frac{s}{2}-1}; q^{\frac{s}{2}})_{k-1} \prod_{t=0}^{s/2-2}(1+zbq^{s(k-1)/2+1} q^{t}) \\
&=& (-zbq, \dots,-zbq^{\frac{s}{2}-1}; q^{\frac{s}{2}})_{k-1} (-zbq^{\frac{s(k-1)}{2}+1};q)_{\frac{s}{2}-1},
\end{eqnarray*}
and
\begin{eqnarray*}
(-bq, \dots,-bq^{\frac{s}{2}-1}; q^{\frac{s}{2}})_{k} &=& (-bq, \dots,-bq^{\frac{s}{2}-1}; q^{\frac{s}{2}})_{k-1} \prod_{t=1}^{s/2-1}(1+bq^{t}q^{s(k-1)/2})\\
&=& (-bq, \dots,-bq^{\frac{s}{2}-1}; q^{\frac{s}{2}})_{k-1} \prod_{t=0}^{s/2-2}(1+bq^{s(k-1)/2+1} q^{t})\\
&=& (-bq, \dots,-bq^{\frac{s}{2}-1}; q^{\frac{s}{2}})_{k-1} (-bq^{\frac{s(k-1)}{2}+1};q)_{\frac{s}{2}-1}.
\end{eqnarray*}
The sum of the generating functions of all the cases is
\begin{multline*}
A_{k}(q) \biggl\{ z^{k}q^{sk^{2}/2} (-zbq^{s(k-1)/2+1};q)_{s/2-1} (-bq^{s(k-1)/2+1};q)_{s/2-1} \\
+(1-zq^{sk/2})(1-q^{sk/2}) \sum_{i=1}^{s/2-1} bz^{k} q^{sk^{2}/2-i}\biggr\},
\end{multline*}
where
\[
A_{k}(q)= \dfrac{(-zbq,-zbq^{2}, \dots,-zbq^{s/2-1}; q^{s/2})_{k-1}(-bq,-bq^{2}, \dots,-bq^{s/2-1}; q^{s/2})_{k-1}}{(zq^{s/2};q^{s/2})_{k}(q^{s/2};q^{s/2})_{k}}.
\]
Now, the desired series expansion is obtained by summing the above expansion over $k$ and adding one, namely,
\begin{multline*}
\sum_{n, r, l\geq 0}D_{s}(n, r, l)z^{r}b^{l}q^{n}=1+\sum_{k\geq 1} A_{k}(q) \times \\
 \biggl\{ z^{k}q^{sk^{2}/2}(-zbq^{s(k-1)/2+1};q)_{s/2-1} (-bq^{s(k-1)/2+1};q)_{s/2-1} \\
+(1-zq^{sk/2})(1-q^{sk/2}) \sum_{i =1}^{s/2-1} bz^{k}q^{sk^{2}/2-i}\biggr\}.
\end{multline*}
\end{proof}

\begin{example} Alladi's series expansion $\eqref{eqAlladi}$ can be obtained from \textbf{Theorem \ref{podgen}} by setting $s=4$ and using the substitutions $z\rightarrow c$ and $b\rightarrow bc^{-1}$. That is,
\begin{multline*}
\prod\limits_{n\geq 1}\dfrac{(1+zbq^{2n-1})}{(1-zq^{2n})}=1+\sum_{k\geq 1} z^{k}q^{2k^{2}-1} \dfrac{(-zbq;q^{2})_{k-1}(-bq^{2};q)_{k-1}}{(-zq^{2};q^{2})_{k}(q^{2};q^{2})_{k}} \\
\times \biggl\{ q(1+zbq^{2k-1})(1+bq^{2k-1})+b(1-zq^{2k})(1-q^{2k}) \biggr\}.
\end{multline*}
Then, by simplifying the term in the brackets, we get
\[
\prod\limits_{n\geq 1}\dfrac{(1+zbq^{2n-1})}{(1-zq^{2n})}=1+\sum_{k\geq 1} z^{k}q^{2k^{2}-1} \dfrac{(-zbq;q^{2})_{k-1}(-bq^{2};q)_{k-1}}{(-zq^{2};q^{2})_{k}(q^{2};q^{2})_{k}} (b+q)(1+zbq^{4k-1}).
\]
By using the substitutions $z\rightarrow c$ and $b\rightarrow bc^{-1}$, we obtain
\[
\prod\limits_{n\geq 1}\dfrac{(1+bq^{2n-1})}{(1-cq^{2n})}=1+\sum_{k\geq 1} c^{k}q^{2k^{2}-1} \dfrac{(-bq;q^{2})_{k-1}(-bc^{-1}q^{2};q)_{k-1}}{(-cq^{2};q^{2})_{k}(q^{2};q^{2})_{k}} (bc^{-1}+q)(1+bq^{4k-1}).
\]
\end{example}
To delve deeper into the structure of $s$-duplicate partitions, we focus on $D_{s}(n, k)$, the number of $s$-duplicate partitions of $n$ into $k$ parts. Let $D_{4}(n,2)$ denote the number of $4$-duplicate partitions into $2$ parts of $n$. In the Online Encyclopedia of Integer Sequences \cite{r21}, we find that the function $D_{4}(n,2)$ matches the sequence \seqnum{A004524}, where it appears to have several interesting combinatorial interpretations. In the next theorem we use the sum operation on partitions, where figure \ref{youngdiagramofthesum} shows the defintion of the operation using Ferrers diagram of the partitions $\lambda=(5,4,3)$ and $\beta=(4,2)$.

\begin{figure} [h!]
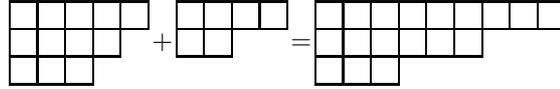

\begin{center}
\ytableausetup
{mathmode, boxframe=normal, boxsize=1em}
\begin{ytableau}

 & & & & & \none[] & & & & & \none[] & & & & & & & & &  \\
 & & & & \none[] & \none[+] & & & \none[] & \none[]  & \none[=] & & & & & &   \\
  & & &  \none[] &  \none[] & \none[] & \none[] & \none[] & \none[] & \none[]  & \none[] & & &   \\
\end{ytableau}
\end{center}
\caption{Ferrers diagram of $\lambda+\beta=(9,6,3)$. \label{youngdiagramofthesum}}
\end{figure}

\begin{theorem} The generating function for the number of $4$-duplicate partitions into $2$ parts is given by
\[
\sum_{n\geq 0} D_{4}(n,2)q^{n}= \dfrac{2q^{4}}{(q^{2};q^{2})_{2}} + \dfrac{q^{3}}{(1-q^{2})^{2}}.
\]
\end{theorem}

\begin{proof} There are three types of partitions in $\mathbb{D}_{4}(n,2)$: partitions into two even parts, partitions into one even part and one odd part, and partitions into two distinct odd parts. The first and the third type satisfy the same generating function with different combinatorial interpretations, which is
\[
\dfrac{q^{4}}{(q^{2};q^{2})_{2}}=q^{4}+q^{6}+2q^{8}+2q^{10}+\cdots.
\]
The function $1/(q^{2};q^{2})_{2}$ generates partitions into at most two even parts. In terms of Ferrers diagram, the exponent on $q^{4}$ consists of the partitions $(2,2)$ and $(3,1)$ in the first and the third type respectively. Then we add the partitions generated by $1/(q^{2};q^{2})_{2}$ to $(2,2)$ and $(3,1)$ to obtain the desired types of partitions. Meanwhile, the second type's partitions are generated by 
\[
\dfrac{q^{2}}{(1-q^{2})} \cdot \dfrac{q}{(1-q^{2})}=q^{3}+2q^{5}+3q^{7}+4q^{9}+\cdots.
\]
By summing up these generating functions, our result follows.
\end{proof}
The following result follows from \textbf{Theorem} \ref{thmduplicate}.
\begin{corollary} We have that $D_{4}(0,2)=D_{4}(1,2)=D_{4}(2,2)=0$, $D_{4}(3,2)=1$, and $D_{4}(4,2)=D_{4}(5,2)=D_{4}(6,2)=2$. For $n\geq 7$, $D_{4}(n,2)$ satisfies the recurrence relation
\[
D_{4}(n,2)=D_{4}(n-4,2)+2.
\]
\end{corollary}

\begin{proof} By listing the first $6$ terms of recurrence $\eqref{4duprec}$ in \textbf{Example \ref{example4duplicate}} for $k=2$, we get
\begin{align*}
&D_{4}(3,2)=D_{4}(1,1)+D_{4}(-1,2)+D_{4}(0,1)=1+0+0=1,\\
&D_{4}(4,2)=D_{4}(2,1)+D_{4}(0,2)+D_{4}(1,1)=1+0+1=2,\\
&D_{4}(5,2)=D_{4}(3,1)+D_{4}(1,2)+D_{4}(2,1)=1+0+1=2,\\
&D_{4}(6,2)=D_{4}(4,1)+D_{4}(2,2)+D_{4}(3,1)=1+0+1=2,\\
&D_{4}(7,2)=D_{4}(5,1)+D_{4}(3,2)+D_{4}(4,1)=1+1+1=3,\\
&D_{4}(8,2)=D_{4}(6,1)+D_{4}(4,2)+D_{4}(5,1)=1+2+1=4.
\end{align*}
We observe that $D_{4}(n,1)=1$ for all $n\geq 1$. Thus, we conclude that $D_{4}(n,2)=D_{4}(n-4,2)+2$ for all $n\geq 7$.
\end{proof}

\section{Partitions into parts simultaneously $s$-congruent and $t$-distinct}
\label{partitionssimcongdist}
A \textit{$t$-distinct partition} $\lambda$ of a positive integer $n$ is a finite sequence of positive integers such that $\lambda_{1}^{u_{1}}+ \lambda_{2}^{u_{2}}+\dots+ \lambda_{k}^{u_{k}}=n$, where $1\leq u_{i}< t$ and $t\geq 2$. Note that in the literature, $t$-distinct partitions are also defined by a gap between parts condition. We shall impose an additional restriction on the set of the $s$-congruent partitions to obtain a new set of partitions into parts simultaneously $s$-congruent and $t$-distinct.
\begin{definition} A partition into parts simultaneously $s$-congruent and $t$-distinct is a partition into parts not congruent to $2,4,6, \dots, (s-2)$ modulo $s$ and appearing fewer than $t$ times. We denote by $C_{s}^{t}(n)$ the number of partitions into parts simultaneously $s$-congruent and $t$-distinct of $n$.
\end{definition}
The generating function for $C_{s}^{t}(n)$ is given by
\begin{equation}
\sum_{n\geq 0} C_{s}^{t}(n)q^{n}=\prod_{n\geq 1} \dfrac{(1-q^{t(2n-1)}) (1-q^{tsn})}{(1-q^{2n-1}) (1-q^{sn})}. \label{gfcongdist}
\end{equation}
For example, $C_{4}^{4}(9)=9$ and the corresponding set is
\[
\mathbb{C}_{4}^{4}(9)=\lbrace (9), (8, 1), (7, 1^{2}), (3^{3}), (3^{2}, 1^{3}), (5, 4), (5, 3, 1), (4, 3, 1^{2}), (4^{2}, 1) \rbrace.
\]
Among the most celebrated identities in the theory of partitions and $q$-series are those of G\"ollnitz and Gordon. These identities were initially discovered by G\"ollnitz \cite{GOL} in 1961, but remained unknown until Gordon \cite{GOR} independently rediscovered them in 1965.
\begin{theorem} (G\"ollnitz-Gordon identities). Fix $a$ to be either $1$ or $3$. Given an integer $n$, the number of partitions of $n$ in which parts are congruent to $4$ or $\pm a$ modulo $8$, is equal to the number of partitions of $n$ in which parts are non-repeating and non-consecutive, with any two even parts differing by at least $4$, and with all parts $\geq a$.
\end{theorem}
In terms of generating functions, these two identities can be written as
\[
\sum_{n\geq 0} \dfrac{q^{n^{2}}(-q;q^{2})_{n}}{(q^{2};q^{2})_{n}}=\dfrac{1}{(q;q^{8})_{\infty} (q^{4};q^{8})_{\infty} (q^{7};q^{8})_{\infty}},
\]
and
\[
\sum_{n\geq 0} \dfrac{q^{n(n+2)}(-q;q^{2})_{n}}{(q^{2};q^{2})_{n}}=\dfrac{1}{(q^{3};q^{8})_{\infty} (q^{4};q^{8})_{\infty} (q^{5};q^{8})_{\infty}}.
\]
Over the years, the G\"ollnitz-Gordon identities have acquired significant attention, akin to the Rogers-Ramanujan identities, the most celebrated identities in the field. Soon thereafter, it led to the major generalization of Andrews, a general theorem \cite[p. 114]{ANDB} that reduces to the G\"ollnitz-Gordon identities in the special cases in which $k=i=2$, $k=i+1=2$.
\begin{theorem} \label{AndrewsgofGG} (Andrews). Let $i$ and $k$ be integers with $0<i\leq k$. Let $V_{k,i}(n)$ denote the number of partitions of $n$ into parts not congruent to $2$ modulo $4$ and not congruent to $0,\pm (2i-1)$ modulo $4k$. Let $W_{k,i}(n)$ denote the number of partitions $(\lambda_{1}^{u_{1}}, \lambda_{2}^{u_{2}}, \dots, \lambda_{m}^{u_{m}})$ of $n$ in which no odd part is repeated, $\lambda_{j}\geq \lambda_{j+1}$, $\lambda_{j}- \lambda_{j+k-1}\geq 2$ if $\lambda_{j}$ odd, $\lambda_{j}- \lambda_{j+k-1}> 2$ if $\lambda_{j}$ even, and at most $i-1$ parts are $\leq 2$. Then
\[
V_{k,i}(n)=W_{k,i}(n).
\]
\end{theorem}

In further exploration of Andrews' general theorem \cite[p. 114]{ANDB}, we enlarge its scope by expanding the class of partitions enumerated by $V_{k,i}(n)$ for certain values of $k$ and $i$.

\begin{definition} \label{defpartofAndthm} Let $t\geq 3$ be a positive integer with $t \not\equiv 2, 4, 6, \dots, (s-2)$ $(\mathrm{mod} \ s)$. Let $E_{s}^{t}(n)$ denote the number of partitions of $n$ into parts not congruent to $ 2, 4, 6, \dots, (s-2)$ modulo $s$ and not congruent to $0, t(2r+1)$ modulo $ts$, where $r=0, 1, 2, 3, \dots, s/2-1$. 
\end{definition}

\begin{theorem} The generating function for $E_{s}^{t}(n)$ is given by
\[
\sum_{n\geq 0} E_{s}^{t}(n)q^{n} = \dfrac{(q^{2};q^{2})_{\infty} (q^{t};q^{t})_{\infty} (q^{ts};q^{ts})_{\infty}}{(q;q)_{\infty} (q^{s};q^{s})_{\infty} (q^{2t};q^{2t})_{\infty}}.
\]
\end{theorem}

\begin{proof} The proof is straightforward and involves elementary generating function manipulations. Consider the positive integer $t\geq 3$ with $t \not\equiv 2, 4, 6, \dots, (s-2)$ (mod $s$). The generating function 
\begin{equation}
\dfrac{(q^{2};q^{2})_{\infty}}{(q;q)_{\infty} (q^{s};q^{s})_{\infty}} \label{notcongtos}
\end{equation}
generates partitions into parts $\lambda_{i} \not\equiv 2, 4, 6, \dots, (s-2)$ (mod $s$). Now, for $r=0, 1, 2, 3, \dots, s/2-1$, consider the product of generating functions
\begin{equation}
(q^{t};q^{ts})_{\infty} (q^{3t};q^{ts})_{\infty} (q^{5t};q^{ts})_{\infty} \cdots (q^{t(s-1)};q^{ts})_{\infty}. \label{gfprod}
\end{equation}
For $n\geq 0$, the exponents on $q$ in the product $\eqref{gfprod}$ are of the form
\begin{gather*} 
t(sn+1)\\
t(sn+3)\\
t(sn+5)\\
\vdots \\
t(sn+s-1).\\ 
\end{gather*}
Since $s\geq 4$ is even and $t\geq 3$, we obtain
\[
(q^{t};q^{ts})_{\infty} (q^{3t};q^{ts})_{\infty} (q^{5t};q^{ts})_{\infty} \cdots (q^{t(s-1)};q^{ts})_{\infty}=\prod_{n\geq 1}(1-q^{t(2n-1)}),
\]
therefore, 
\begin{equation}
(q^{t};q^{ts})_{\infty} (q^{3t};q^{ts})_{\infty} (q^{5t};q^{ts})_{\infty} \cdots (q^{t(s-1)};q^{ts})_{\infty}(q^{ts};q^{ts})_{\infty}=\dfrac{(q^{t};q^{t})_{\infty}}{(q^{2t};q^{2t})_{\infty}}(q^{ts};q^{ts})_{\infty}. \label{congtost}
\end{equation}
By multiplying the generating functions $\eqref{notcongtos}$ and $\eqref{congtost}$ together, our result follows.
\end{proof}

\begin{example} For $(s,t,n)=(6,3,14)$, there are $13$ partitions into parts not congruent to $2,4$ modulo $6$ and $0,3,9,15$ modulo $18$, such that
\begin{multline*}
\mathbb{E}_{6}^{3}(14)=\lbrace (13,1), (12, 1^{2}), (11,3), (7^{2}), (7,6,1), (7,5,1^{2}), (7, 1^{7}),\\ 
(6^{2},1^{2}), (6, 1^{8}), (6, 5, 1^{3}), (5^{2}, 1^{4}), (5, 1^{9}), (1^{14}) \rbrace.
\end{multline*}
\end{example}

\begin{theorem} Let $t\geq 3$ be a positive integer with $t \not\equiv 2, 4, 6, \dots, (s-2)$ modulo $s$. Then for every positive integer $n\geq 0$,
\[
C_{s}^{t}(n)=E_{s}^{t}(n).
\]
\end{theorem}

\begin{proof}  By employing the fact $(q;q^{2})^{-1}_{\infty}=(q^{2};q^{2})_{\infty}/(q;q)_{\infty}$ in $\eqref{gfcongdist}$, we obtain
\begin{eqnarray*}
\sum_{n\geq 0} C_{s}^{t}(n)q^{n}&=& \dfrac{(q^{t};q^{2t})_{\infty}(q^{ts};q^{ts})_{\infty}}{(q;q^{2})_{\infty} (q^{s};q^{s})_{\infty}}\\
&=& \dfrac{(q^{2};q^{2})_{\infty} (q^{t};q^{t})_{\infty} (q^{ts};q^{ts})_{\infty}}{(q;q)_{\infty} (q^{2t};q^{2t})_{\infty} (q^{s};q^{s})_{\infty}}= \sum_{n\geq 0} E_{s}^{t}(n)q^{n}.
\end{eqnarray*}
\end{proof}

The relation between the number $E^{t}_{s}(n)$ and Andrews' theorem \cite[p. 114]{ANDB} becomes apparent by setting $s=4$ in \textbf{Definition $\ref{defpartofAndthm}$}. We have $E^{t}_{4}(n)=V_{t,(t+1)/2}(n)$ for every odd $t\geq 3$, where $t=k=2i-1$. In essence, this establishes a relationship with Andrews' theorem and highlights the significance of these class of partitions in the context of the study.

\begin{corollary} \label{corequi} Let $t\geq 3$ be an odd integer. Then for every natural number $n\geq0$,
\[
C_{4}^{t}(n) = V_{t,(t+1)/2}(n) = W_{t,(t+1)/2}(n).
\]
\end{corollary}
Note that, when $t=3,5$, the generating function
\[
\sum_{n\geq 0} C_{4}^{t}(n)q^{n} = \dfrac{(q^{2};q^{2})_{\infty} (q^{t};q^{t})_{\infty} (q^{4t};q^{4t})_{\infty}}{(q;q)_{\infty} (q^{4};q^{4})_{\infty} (q^{2t};q^{2t})_{\infty}}=\dfrac{\psi(-q^{t})}{\psi(-q)}
\]
is the generating function for the sequences \seqnum{A036018} and \seqnum{A036026}, respectively, in the Online Encyclopedia of Integer Sequences \cite{r21}.

\begin{table}[h!]
\caption{An illustration of Corollary $\ref{corequi}$ for $(s,t,n)=(4,3,12)$.}
\begin{tabular}{@{}ccccc@{}}
\toprule
$\mathbb{C}_{4}^{3}(12)$     &  & $\mathbb{V}_{3,2}(12)$       &  & $\mathbb{W}_{3,2}(12)$  \\
\midrule
$(12)$    &                   & $(1^{12})$       &                   & $(12)$       \\
$(11,1)$    &                   & $(11,1)$       &                   & $(11,1)$       \\
$(8,4)$    &                   & $(8,4)$       &                   & $(8,4)$       \\
$(8,3,1)$    &                   & $(8,1^{4})$       &                   & $(8,3,1)$       \\
$(7,5)$    &                   & $(7,5)$       &                   & $(7,5)$       \\
$(7,4,1)$    &                   & $(7,4,1)$       &                   & $(7,4,1)$       \\
$(7,3,1^{2})$    &                   & $(7,1^{5})$       &                   & $(7,3,2)$       \\
$(5^{2},1^{2})$    &                   & $(5^{2},1^{2})$       &                   & $(10,2)$       \\
$(5,4,3)$    &                   & $(5,4,1^{3})$       &                   & $(5,4,3)$       \\
$(5,3^{2},1)$    &                   & $(5,1^{7})$       &                   & $(6,5,1)$       \\
$(9,3)$    &                   & $(4^{3})$       &                   & $(9,3)$       \\
$(4,3^{2},1^{2})$    &                   & $(4,1^{8})$       &                   & $(6,4,2)$       \\
$(4^{2},3,1)$    &                   & $(4^{2},1^{4})$       &                   & $(6^{2})$       \\
\botrule
\end{tabular}
\end{table}

\section{$s$-Modular, $s$-congruent and $s$-duplicate overpartitions}
\label{overpartitions}
An overpartition of a nonnegative integer $n$ is a nonincreasing sequence of natural numbers whose sum is $n$, and where the first occurrence of a number may be overlined. They were introduced by Corteel and Lovejoy \cite{COR}. The generating function of overpartition is
\[
\sum_{n\geq 0} \overline{p}(n)q^{n}= \dfrac{(-q;q)_{\infty}}{(q;q)_{\infty}}.
\]
where $\overline{p}(n)$ denotes the number of overpartitions of $n$. Let $\overline{M}_{s}(n)$, $\overline{C}_{s}(n)$, and $\overline{D}_{s}(n)$ denote the number of $s$-modular, $s$-congruent, and $s$-duplicate overpartitions of $n$, respectively.

\begin{theorem} \label{thmovermodular} The generating function for $\overline{M}_{s}(n)$, the number of $s$-modular overpartitions of $n$, is given by
\[
\sum_{n\geq 0} \overline{M}_{s}(n) q^{n} = \dfrac{(-2q;q)_{\infty}(-q^{s};q^{s})_{\infty}}{(q^{s};q^{s})_{\infty}}=\dfrac{(-2q;q)_{\infty}(q^{2s};q^{2s})_{\infty}}{(q^{s};q^{s})^{2}_{\infty}}.
\]
\end{theorem}

\begin{proof} The generating function for the number of overpartitions where the multiplicity of every part is congruent to $0$ modulo $s$ is given by
\[
\dfrac{(-q^{s};q^{s})_{\infty}}{(q^{s};q^{s})_{\infty}},
\]
and on the other hand, the number of overpartitions where the multiplicity of every part is congruent to $1$ modulo $s$ has generating function
\[
(-2q;q)_{\infty}.
\]
Then, by multiplying both generating functions together, our result follows.
\end{proof}

\begin{theorem} \label{thmovercongruent} The generating function for $\overline{C}_{s}(n)$, the number of $s$-congruent overpartitions of $n$, is given by
\[
\sum_{n\geq 0} \overline{C}_{s}(n) q^{n} = \dfrac{(-q;q^{2})_{\infty}(-q^{s};q^{s})_{\infty}}{(q;q^{2})_{\infty}(q^{s};q^{s})_{\infty}}=\dfrac{(q^{2};q^{2})^{3}_{\infty}(q^{2s};q^{2s})_{\infty}}{(q;q)^{2}_{\infty}(q^{4};q^{4})_{\infty}(q^{s};q^{s})^{2}_{\infty}}.
\]
\end{theorem}

\begin{proof} The number of overpartitions into odd parts has the generating function
\[
\dfrac{(-q;q^{2})_{\infty}}{(q;q^{2})_{\infty}},
\]
while
\[
\dfrac{(-q^{s};q^{s})_{\infty}}{(q^{s};q^{s})_{\infty}}
\]
generates overpartitions into parts congruent to $0$ modulo $s$.
\end{proof}

\begin{theorem} \label{thmoverduplicate} The generating function for $\overline{D}_{s}(n)$, the number of $s$-duplicate overpartitions of $n$, satisfies the identity
\[
\sum_{n\geq 0} \overline{D}_{s}(n) q^{n} = \dfrac{(-2q;q)_{\infty}(-q^{s/2};q^{s/2})_{\infty}}{(-2q^{s/2};q^{s/2})_{\infty} (q^{s/2};q^{s/2})_{\infty}}=\dfrac{(-2q;q)_{\infty}(q^{s};q^{s})_{\infty}}{(-2q^{s/2};q^{s/2})_{\infty} (q^{s/2};q^{s/2})^{2}_{\infty}}.
\]
\end{theorem}

\begin{proof} The generating function
\[
\dfrac{(-2q;q)_{\infty}}{(-2q^{s/2};q^{s/2})_{\infty}}
\]
generates overpartitions into distinct parts indivisible by $s/2$, and
\[
\dfrac{(-q^{s/2};q^{s/2})_{\infty}}{(q^{s/2};q^{s/2})_{\infty}}
\]
generates overpartitions into parts congruent to $0$ modulo $s/2$.
\end{proof}

\section*{Concluding remarks and perspectives}

In this paper, we have explored three equinumerous classes of partitions: $s$-modular, $s$-congruent, and $s$-duplicate partitions. Specifically, we established a bijection between these classes and investigated several of their combinatorial properties. We also examined the relationship between partitions into parts that are simultaneously $s$-congruent and $t$-distinct and the Andrews-Göllnitz-Gordon theorem. These results contribute significantly to the literature on integer partitions and provide a fertile ground for ongoing and future research.

Future research could address several intriguing questions:

(1) Are there any series or product expansions using the Ramanujan theta-function for $D_{s}(n)$ for $s \geq 6$, analogous to the expansion $\sum_{n\geq 0} \mypod(n)=\frac{1}{\psi(-q)}$?

(2) What is the impact of $s$-duplicate partitions on topics where $\mypod(n)$ appears?

(3) Are there any recurrence relations, similar to that of $C_{4}(n)$, applicable to the remaining cases?

(4) What interesting arithmetic properties do these classes of partitions exhibit?

(5) Are there any dissections for the generating function $\frac{f_2}{f_1 f_s}$ for $s\geq 6$?

(6) In the context of $D_{4}(n,2)$, are there any interesting combinatorial interpretations for $D_{4}(n,k)$ for $k\geq 3$? 

(7) The same previous question applies to $M_{s}(n,k)$ and $C_{s}(n,k)$?

(8) Are there any equivalent classes to $\mathbb{E}_{s}^{t}(n)$ that satisfy specific difference conditions?

(9) What arithmetic properties do $E_{s}^{t}(n)$ exhibit?

(10) Is there an overpartition analogue for the partitions enumerated by $E_{s}^{t}(n)$?

(11) What arithmetic properties do $\overline{C}_{s}(n)$ exhibit?

\bmhead{Acknowledgements} The authors would like to thank the reviewers for their valuable remarks and suggestions to improve the original manuscript. This work was supported by DG-RSDT (Algeria), PRFU Project, No. C00L03UN180120220002. 

\section*{Conflict of interest statement}

On behalf of all authors, the corresponding author states that there is no conflict of interest.
\section*{Data availability statement}
On behalf of all authors, the corresponding author states that the manuscript has no associated data.


\end{document}